\documentclass[a4paper,11pt, british]{amsart}

\usepackage{babel}
\usepackage{enumitem}
\usepackage{hyperref}
\usepackage[utf8]{inputenc}
\usepackage{newunicodechar}
\usepackage{mathtools}
\usepackage{varioref}
\usepackage[centering]{geometry}
\usepackage[arrow,curve,matrix]{xy}

\usepackage{colortbl}
\usepackage{graphicx}
\usepackage{tikz}
\usepackage{tikz-cd}

\usepackage{mathrsfs}

\usepackage{longtable} 
\usepackage{caption}
\newcolumntype{M}[1]{>{\centering\arraybackslash}m{#1}} 

\definecolor{linkred}{rgb}{0.7,0.2,0.2}
\definecolor{linkblue}{rgb}{0,0.2,0.6}

\numberwithin{figure}{section}

\setdescription{labelindent=\parindent, leftmargin=2\parindent}
\setitemize[1]{labelindent=\parindent, leftmargin=2\parindent}
\setenumerate[1]{labelindent=0cm, leftmargin=*, widest=iiii}

\DeclareFontFamily{OMS}{rsfs}{\skewchar\font'60}
\DeclareFontShape{OMS}{rsfs}{m}{n}{<-5>rsfs5 <5-7>rsfs7 <7->rsfs10 }{}
\DeclareSymbolFont{rsfs}{OMS}{rsfs}{m}{n}
\DeclareSymbolFontAlphabet{\scr}{rsfs}
\DeclareSymbolFontAlphabet{\scr}{rsfs}

\DeclareMathOperator{\Aut}{Aut}
\DeclareMathOperator{\codim}{codim}

\DeclareMathOperator{\Pic}{Pic}

\DeclareMathOperator{\red}{red}
\DeclareMathOperator{\reg}{reg}

\DeclareMathOperator{\Sym}{Sym}
\DeclareMathOperator{\supp}{supp}

\DeclareMathOperator{\Bl}{Bl}
\DeclareMathOperator{\Gr}{Gr}
\DeclareMathOperator{\gr}{gr}
\DeclareMathOperator{\Supp}{Supp}
\DeclareMathOperator{\mult}{mult}

\newcommand{\sI}{\scr{I}}

\newcommand{\sO}{\scr{O}}

\newcommand{\cC}{\mathcal C}
\newcommand{\cD}{\mathcal D}

\newcommand{\cH}{\mathcal H}
\newcommand{\cI}{\mathcal I}
\newcommand{\cK}{\mathcal K}
\newcommand{\cM}{\mathcal M}

\newcommand{\cS}{\mathcal S}
\newcommand{\cU}{\mathcal U}
\newcommand{\cV}{\mathcal V}

\newcommand{\cX}{\mathcal X}

\newcommand{\bA}{\mathbb{A}}

\newcommand{\bC}{\mathbb{C}}

\newcommand{\bG}{\mathbb{G}}
\newcommand{\bH}{\mathbb{H}}

\newcommand{\bL}{\mathbb{L}}

\newcommand{\bO}{\mathbb{O}}
\newcommand{\bP}{\mathbb{P}}
\newcommand{\bQ}{\mathbb{Q}}
\newcommand{\bR}{\mathbb{R}}
\newcommand{\bS}{\mathbb{S}}

\newcommand{\bX}{\mathbb{X}}

\newcommand{\bZ}{\mathbb{Z}}

\newcommand{\fD}{\mathfrak{D}}
\newcommand{\fg}{\mathfrak{g}}
\newcommand{\fh}{\mathfrak{h}}
\newcommand{\fH}{\mathfrak{H}}

\newcommand{\fn}{\mathfrak{n}}
\newcommand{\fp}{\mathfrak{p}}
\newcommand{\fq}{\mathfrak{q}}

\theoremstyle{plain}
\newtheorem{thm}{Theorem}[section]

\newtheorem{cor}[thm]{Corollary}
\newtheorem{defn}[thm]{Definition}

\newtheorem{lem}[thm]{Lemma}

\newtheorem{prop}[thm]{Proposition}

\theoremstyle{remark}

\newtheorem{c-n-d}[thm]{Claim and Definition}

\newtheorem{example}[thm]{Example}

\newtheorem{notation}[thm]{Notation}

\newtheorem{rem}[thm]{Remark}
\newtheorem{question}[thm]{Question}
\newtheorem*{rem-nonumber}{Remark}

\numberwithin{equation}{thm}

\setlist[enumerate]{label=(\thethm.\arabic*), before={\setcounter{enumi}{\value{equation}}}, after={\setcounter{equation}{\value{enumi}}}}

\newcommand{\factor}[2]{\left. \raise 2pt\hbox{$#1$} \right/\hskip -2pt\raise -2pt\hbox{$#2$}}
\author{Jie Liu} %
\address{Jie Liu, Institute of Mathematics, Academy of Mathematics and Systems Science, Chinese Academy of Sciences, Beijing, 100190, China}
\email{\href{jliu@amss.ac.cn}{jliu@amss.ac.cn}}
\urladdr{\href{http://www.jliumath.com}{http://www.jliumath.com}}

\keywords{$G$-variety, Fano manifold, tangent bundle, moment map, VMRT}

\makeatletter
\@namedef{subjclassname@2020}{2020 Mathematics Subject Classification}
\makeatother
\subjclass[2020]{14M17,14M27,14J45}

\title{On moment map and bigness of tangent bundles of $G$-varieties}
\date{\today}

\makeatletter

\hypersetup{
	pdfauthor={\authors},
	pdftitle={\@title},
	pdfsubject={\@subjclass},
	pdfkeywords={\@keywords},
	pdfstartview={Fit},
	pdfpagemode={UseOutlines},
	pdfpagelayout={OneColumn},
	bookmarks,
	colorlinks,
	linkcolor=linkblue,
	citecolor=linkred,
	urlcolor=linkred
}
\makeatother

\DeclareMathOperator{\GL}{GL}

\DeclareMathOperator{\Rad}{Rad}

\DeclareMathOperator{\SL}{SL}
\DeclareMathOperator{\Eff}{\overline{Eff}}
\DeclareMathOperator{\Proj}{Proj}

\begin{document}
	
\begin{abstract}
	Let $G$ be a connected algebraic group and let $X$ be a smooth projective $G$-variety. In this paper, we prove a sufficient criterion to determine the bigness of the tangent bundle $TX$ using the moment map $\Phi_X^G:T^*X\rightarrow \fg^*$. As an application, the bigness of the tangent bundles of certain quasi-homogeneous varieties are verified, including symmetric varieties, horospherical varieties and equivariant compactifications of commutative linear algebraic groups. Finally, we study in details the Fano manifolds $X$ with Picard number $1$ which is an equivariant compactification of a vector group $\bG_a^n$. In particular, we will determine the pseudoeffective cone of $\bP(T^*X)$ and show that the image of the projectivised moment map along the boundary divisor $D$ of $X$ is projectively equivalent to the dual variety of the VMRT of $X$.
\end{abstract}

\maketitle
\tableofcontents

\section{Introduction}

Throughout we work over the field of complex numbers $\bC$.

Since the seminal works of Mori and Siu-Yau on the solutions to Hartshorne conjecture and Frankel conjecture \cite{Mori1979,SiuYau1980}, it becomes apparent that making an assumption about the positivity of the tangent bundle $TX$ of a projective manifold $X$, or equivalently the positivity of the tautological divisor $\Lambda$ of the projectivisation $\bP(T^*X)$ (in the geometric sense), allows us to derive a particularly rich geometry of $X$. While the situation where $TX$ is ample or nef are intensively studied in literatures (see \cite{Mori1979,CampanaPeternell1991,DemaillyPeternellSchneider1994,MunozOcchettaSolaCondeWatanabeEtAl2015} and the references therein), the case where $TX$ is big is much less understood yet. The main difficult to investigate the bigness of $TX$ in general case is the lack of numerical characterisations in terms of invariants of $X$ even in low dimension. As far as we know, there are three main tools which are used to prove or disprove the bigness of $TX$. The first one is the \emph{(projectivised) moment map}, i.e., the rational map defined by certain subspaces of $|\Lambda|$. The second one is the existence of \emph{twisted symmetric vector fields}, i.e., non-vanishing of $H^0(X,\Sym^m T_X\otimes \sO_X(-A))$ with $A$ being a big divisor, and the third one is to determine the cohomological class of the \emph{total dual VMRT} $\check{\cC}\subseteq \bP(T^*X)$. We give in the following an overview of a few of the varieties which have already been studied and also the method used to prove or disprove the bigness of their tangent bundles:

\begin{itemize}
	\item (projectivised) moment map 
	\begin{itemize}
		\item rational homogeneous spaces \cite{Richardson1974}
	\end{itemize}
	
	\item twisted symmetric vector fields
	\begin{itemize}
		\item toric varieties \cite{Hsiao2015}
		\item intersection of two quadrics in $\bP^4$ and cubic surfaces in $\bP^3$ \cite{Mallory2021}
		\item hypersurfaces in $\bP^n$ $(n\geq 3)$ \cite{HoeringLiuShao2020}
	\end{itemize}	 
	
	\item total dual VMRT
	\begin{itemize}
		\item del Pezzo surfaces and del Pezzo threefolds \cite{HoeringLiuShao2020}
		\item Fano manifolds of Picard number $1$ and with $0$-dimensional VMRT \cite{HoeringLiu2021}
		\item moduli spaces $\text{SU}_C(r,d)$ of stable vector bundles of rank $r$ and degree $d$ over a projective curve $C$ of genus $g$ such that $r\geq 3$, $g\geq 4$ and $(r,d)=1$ \cite{FuLiu2021}
		\item Fano threefolds with Picard number $2$ \cite{KimKimLee2022}
	\end{itemize} 
\end{itemize}

The main body of this paper will be devoted to pursue furthermore the criterion for the bigness of $TX$ via moment map. Let $G$ be a connected algebraic group with Lie algebra $\fg$ and let $X$ be a smooth projective $G$-variety. Then the \emph{moment map} $\Phi_X^G:T^*X\rightarrow \fg^*$ is defined as follows: for a point $x\in X$, the map $T^*_x X\rightarrow \fg^*$ is defined as the cotangent map of the orbit map $\mu_x: G\rightarrow Gx$ at $x$, see Section\,\ref{subsection:momentmap} for more details. We denote by $\cM_X^G\subseteq \fg^*$ the closure of the image of $\Phi_X^G$. The starting point of this paper is the following criterion for the bigness of $TX$, which is proved by combining the moment map with the approach via twisted symmetric vector fields.

\begin{prop}
	\label{p.Intro-criterion}
	Let $G$ be a connected algebraic group and let $X$ be a smooth projective $G$-variety. Then $TX$ is big if there exists an effective big divisor $A$ such that 
	\[
	\dim(\Phi_X^G(T^*X|_{\Supp(A)}))<\dim(\cM_X^G).
	\]
\end{prop}

We refer the reader to Section \ref{subsection:criterion} for discussion on how to verify the conditions in the criterion. As the first application of Proposition \ref{p.Intro-criterion}, the following theorem confirms the bigness of the tangent bundles of certain interesting smooth projective quasi-homogeneous varieties.

\begin{thm}
	\label{t.Examples-main-thm}
	Let $X$ be a projective manifold. Then $TX$ is big if $X$ is isomorphic to one of the following varieties:
	\begin{enumerate}
		\item a spherical $G$-variety with a $G$-stable affine open subset, e.g., symmetric varieties;
		
		\item a horospherical $G$-variety;
		
		\item\label{t.term-EC} a quasi-homogeneous $G$-variety with $G$ a commutative linear algebraic group.
	\end{enumerate}
\end{thm}

We refer the reader to Section \ref{subsection.proofmainthm} for the definitions of spherical varieties, symmetric varieties and horospherical varieties. Our initial motivation for the present work is trying to produce more examples of Fano manifolds of Picard number $1$ and with big tangent bundle, while the only previous known non-homogeneous examples, up to our knowledge, are the quintic del Pezzo threefold $V_3$ \cite{HoeringLiuShao2020} and the horospherical $G_2$-variety $X^5$ \cite{PasquierPerrin2010}. As the second application of Proposition \ref{p.Intro-criterion}, we derive infinitely many (non-homogeneous) examples of Fano manifolds of Picard number $1$ and with big tangent bundle, which are summarised in the following:

	\begin{itemize}
		\item rational homogeneous spaces $G/P$ with Picard number $1$ \cite{Richardson1974}. See Theorem \ref{t.KnopdimensionofMomentmap} and Example \ref{e.Examples-Spherical-Varieties}
		
		\item non-homogeneous projective symmetric varieties (and their degenerations) \cite{Ruzzi2010}
		\begin{itemize}
			\item the Cayley Grassmannian $CG$ \cite{Manivel2018}
			\item the double Cayley Grassmannian $DG$ \cite{Manivel2020a}
			\item a smooth hyperplane section of the third row of the \emph{geometric Freudenthal's magic square} $\Gr_{\omega}(\bA^3,\bA^6)$, where $\bA$ is a complex composition algebra (i.e., the complexification of $\bR$, $\bC$, the quaternions $\bH$, or the octonions $\bO$) \cite{LandsbergManivel2001}
		\end{itemize}
		See Corollary \ref{c.symmetric-Picard-number-one} and Remark \ref{r.Classification-Symmetric-varieties}.                                                        
		
		\item non-homogeneous smooth projective horospherical varieties \cite{Pasquier2009}
		\begin{itemize}
			\item $X^1(m)\coloneqq (B_m,\omega_{m-1},\omega_m)$ $(m\geq 3)$
			\item $X^2\coloneqq (B_3,\omega_1,\omega_3)$
			\item $X^3(m,i)\coloneqq (C_m,\omega_i,\omega_{i+1})$ $(m\geq 2, 1\leq i\leq m-1)$
			\item $X^4\coloneqq (F_4,\omega_2,\omega_3)$
			\item $X^5\coloneqq (G_2,\omega_2,\omega_1)$
		\end{itemize}  
	    The varieties $X^3(m,i)$ are the odd symplectic Grassmannians and we refer the reader to \cite{Pasquier2009} for the notations. See Proposition \ref{p.horospherical} and Remark \ref{r.Examples-Horospherical}.
		
		\item a smooth linear section $V_k$ of $\Gr(2,5)\subseteq \bP^9$ with codimension $k\leq 3$ in its P\"ucker embedding \cite{HoeringLiuShao2020,HoeringLiu2021}. The variety $V_1$ is isomorphic to the horospherical variety $X^3(2,1)$. See Proposition \ref{p.commutative-groups} and Example \ref{e.Examples-EC-Picard-number-one}.
		
		\item a smooth linear section $S_k$ of the $10$-dimensional spinor variety $\bS_5\subseteq \bP^{15}$ with codimension $k\leq 3$ in its minimal embedding. The variety $S_1$ is the horospherical variety $X^2$. See Proposition \ref{p.commutative-groups}, Exmaple \ref{e.Examples-EC-Picard-number-one} and Corollary \ref{c.linear-section-S10}.
		
		\item the smooth projective two-orbits $F_4$-variety $\textbf{X}_1$ given in \cite[Definition 2.11]{Pasquier2009}. Note that the smooth projective two-orbits $G_2\times \text{PGL}_2$-variety $\textbf{X}_2$ is isomorphic to the general codimension $2$ linear section $S_2^g$ of $\bS_5$. See \cite[Definition 2.12]{Pasquier2009}, \cite[Proposition 4.8]{BaiFuManivel2020} and Proposition \ref{p.X1}.
	\end{itemize}

An interesting class of examples belonging to the case \ref{t.term-EC} of Theorem \ref{t.Examples-main-thm} is the equivariant compactifications of vector groups. Recall that an \emph{equivariant compactification} of an algebraic group $G$ is a pair $(X,x)$, where $X$ is a normal complete algebraic variety equipped with a regular action $G\times X\rightarrow X$ and $x\in X$ is a point with the trivial stabilizer such that the orbit $Gx$ is open and dense in $X$. We find it quite remarkable that the moment map of a Fano manifold $X$ with Picard number $1$ which is an equivariant compactification of a vector group $\bG_a^n$ exhibits many interesting geometric properties and it has a surprising connection with the VMRT of $X$. In particular, in this situation, as the last application of Proposition \ref{p.Intro-criterion}, we have a complete description of the pseudoeffective cone of $\bP(T^*X)$ and we can relate the criterion given in Proposition \ref{p.Intro-criterion} to the criterion given by total dual VMRT in \cite{HoeringLiuShao2020,FuLiu2021}, see also Section \ref{subsection.VMRTdual}.

\begin{thm}
	\label{t.EC-main-thm}
	Let $X$ be a Fano manifold with Picard number $1$, different from projective spaces, which is an equivariant compactification of the vector group $G=\bG_a^n$ with an open orbit $O\subseteq X$. Denote by $D$ the complementary $X\setminus O$. Then the following statements hold.
	\begin{enumerate}
		\item The pseudoeffective cone $\Eff(\bP(T^*X))$ is generated by $\pi^*D$ and all the prime divisors $D_{\cH}$ (cf. Notation \ref{n.Horizontal-part}), where $\pi:\bP(T^*X)\rightarrow X$ is the natural projection and $\cH$ is a reduced and irreducible hypersurface in $\bP(\fg^*)$ containing $\bP(\cM_D^{G})$.
		
		\item\label{item.EC-main-thm} If the VMRT $\cC_x\subseteq \bP(T_x X)$ at a point $x\in O$ is smooth, then $\bP(\cM_D^{G})\subseteq \bP(\fg^*)$ is projectively equivalent to the dual variety of $\cC_x\subseteq \bP(T_x X)$.
		
		\item If the VMRT $\cC_x\subseteq \bP(T_x X)$ at a point $x\in O$ is smooth and not dual defective, then we have
		\begin{center}
			$D_{\cH}=\check{\cC}$,\quad $\Eff(\bP(T^*X))=\langle D_{\cH}, \pi^*D\rangle$\quad and\quad $D_{\cH}\sim a\Lambda - 2\pi^*D$
		\end{center}
		where $\cH=\bP(\cM_D^{G})\subseteq \bP(\fg^*)$, the variety $\check{\cC}\subseteq \bP(T^*X)$ is the total dual VMRT, the coefficient $a$ is the codegree of the VMRT and $\Lambda$ is the tautological divisor of $\bP(T^*X)$.
	\end{enumerate}
\end{thm}

\begin{rem}
	\begin{enumerate}
		\item For projective spaces, there exist non-isomorphic equivariant compactifications structures of vector groups and they are classified by the so-called \emph{Hasset-Tschinkel correspondence} proved in \cite{HassettTschinkel1999} (see also \cite{FuHwang2014}). In particular, Theorem \ref{t.EC-main-thm} above holds for the simplest equivariant compactification structure on projective spaces (see Example \ref{e.Examples-EC-structure}), however the statement \ref{item.EC-main-thm} is no longer true for others.
		
		\item For the known examples of equivariant compactifications of vector groups with Picard number $1$ (cf. Example \ref{e.Examples-EC-Picard-number-one}), we will determine in Table \ref{table.codegree-EC} the dual defect and the codegree of their VMRT, i.e., the value of $a$.
\end{enumerate}
\end{rem}

\subsection*{Acknowledgements} 

I would like to thank Baohua Fu for his stimulating discussion and useful comments during this project. In particular, the proof of Proposition \ref{p.boundary-divisor-EC} is communicated to me by him.
This work is supported by the National Key R\&D Program of China (No. 2021YFA1002300) and the NSFC grants (No. 11688101 and No. 12001521).

\section{Notation, conventions, and facts used}

Let $X$ be a projective manifold. Denote by $N^1(X)_{\mathbb{R}}$ the finite-dimensional $\mathbb{R}$-vector space of numerical equivalence classes of $\mathbb{R}$-divisors. The \emph{pseudoeffective cone} $\Eff(X)\subseteq N^1(X)_{\bR}$ is the closure of the convex cone spanned by the classes of effective $\bR$-divisors. It is known the interior of $\Eff(X)$ is the \emph{big cone} $\text{Big(X)}$ of $X$; that is, the open cone generated by big divisors on $X$.

\subsection{Positivity of vector bundles}

Let $E$ be a vector bundle over a smooth projective variety $X$. We denote by $\bP(E)$ the natural (not the Grothendieck) projectivisation of $E$; that is, we have
\[
\bP(E)\coloneqq \Proj\left(\bigoplus_{m\geq 0} \text{Sym}^m E^*\right),
\]
where $E^*$ is the dual bundle of $E$. 

\begin{defn}
	Let $E$ be a vector bundle over a projective manifold $X$. We say that $E$ is big (resp. ample, nef, pseudoeffective) if the tautological line bundle $\sO_{\bP(E^*)}(1)$ of the projective bundle $\bP(E^*)$ is big (resp. ample, nef, pseudoeffective).
\end{defn}

\begin{example}
	Let $E\cong \sO_{\bP^1}(a_1)\oplus \dots \oplus \sO_{\bP^1}(a_r)$ be a vector bundle of rank $r$ over $\bP^1$ such that $a_1\geq \dots \geq a_r$. Then we have
	\[
		E\ \text{is}
		\begin{cases}
			\text{ample}
			    & \text{if and only if}\ a_r>0; \\
			\text{nef} 
			    & \text{if and only if}\ a_r\geq 0; \\
			\text{big}
			    & \text{if and only if}\ a_1>0; \\
			\text{pseudoeffective}
			    & \text{if and only if}\ a_1\geq 0.			
		\end{cases}
	\]
\end{example}

We have the following simple but useful criterion for bigness of vector bundles:

\begin{lem}[\protect{Equivalent definitions of bigness, \cite[Lemma 2.3]{HoeringLiuShao2020}}]
	\label{l.criterionbigness}
	Let $E$ be a vector bundle over a projective manifold $X$. Denote by $\pi:\bP(E^*)\rightarrow X$ the natural projection and by $\Lambda$ the tautological divisor class of $\bP(E^*)$. Then the following statements are equivalent.
	\begin{enumerate}
		\item \label{l.criterionbigness(1)} The vector bundle $E$ is  big.
		
		\item \label{l.criterionbigness(2)} There exist a big divisor $A$ on $X$ and a positive integer $m$ such that $m\Lambda - \pi^*A$ is big.
		
		\item \label{l.criterionbigness(3)} There exist a big divisor $A$ on $X$ and a positive integer $m$ such that $m\Lambda - \pi^*A$ is pseudoeffective.
	\end{enumerate}
\end{lem}

\begin{proof}
	The implication \ref{l.criterionbigness(1)}$\implies$\ref{l.criterionbigness(2)} follows from the openness of bigness, and the implication \ref{l.criterionbigness(2)}$\implies$\ref{l.criterionbigness(3)} is trivial. Finally the implication \ref{l.criterionbigness(3)}$\implies$\ref{l.criterionbigness(1)} follows from \cite[Lemma 2.3]{HoeringLiuShao2020}.
\end{proof}

\begin{lem}
	\label{l.subbundle-big}
	Let $F$ be a subbundle of $E$. If $F$ is big, then $E$ is big.
\end{lem}

\begin{proof}
	By Lemma \ref{l.criterionbigness}, there exist a big divisor $A$ on $X$ and a positive integer $m$ such that $m\Lambda_F-\pi_F^* A$ is big, where $\Lambda_F$ is the tautological divisor of $\bP(F^*)$ and $\pi_F:\bP(F^*)\rightarrow X$ is the natural projective. Then, after replacing $m$ by its large enough multiple $m'm$ and replacing $A$ by $m'A$, we may assume that $|m\Lambda_F-\pi_F^*A|$ is non-empty. In particular, we have $H^0(X,\Sym^m F\otimes \sO_X(-A))\not=\emptyset$. This implies that $H^0(X,\Sym^m E\otimes \sO_X(-A))\not=\emptyset$. In other words, we have $|m\Lambda_E-\pi_E^*A|\not=\emptyset$ and hence $E$ is big by Lemma \ref{l.criterionbigness}, where $\Lambda_E$ is the tautological divisor of $\bP(E^*)$ and $\pi_E:\bP(E^*)\rightarrow X$ is the natural projection.
\end{proof}

\begin{example}
	Let $X$ be a projective manifold and let $L$ be a big line bundle over $X$. Set $E=L\oplus L^*$. Then the tangent bundle of $\bP(E^*)$ is big. Indeed we note that $E$ is big by Lemma \ref{l.subbundle-big} as $L$ is big . In particular, the relative tangent bundle of $\bP(E^*)\rightarrow X$ is big as it is isomorphic to the line bundle $\sO_{\bP(E^*)}(2)$. Consequently, the tangent bundle $T{\bP(E^*)}$ itself is big by Lemma \ref{l.subbundle-big}. 
\end{example}

\subsection{Pseudoeffective cone of divisors on $G$-varieties}

Let $D$ be a $\bQ$-Weil divisor on a normal projective variety $X$ and let $\bar{D}$ be an irreducible component of $\supp(D)$. We denote by $m_{\bar{D}}(D)$ the multiplicity of $D$ along $\bar{D}$, i.e., the coefficient of $\bar{D}$ in $D$.

\begin{lem}
	\label{l.coefficient-big-divisors}
	Let $X$ be a projective manifold. Let $D$ and $D'$ be two effective $\bQ$-Weil divisors on $X$ such that $\supp(D)=\supp(D')$. Then $D$ is big if and only if $D'$ is big.
\end{lem}

\begin{proof}
	By symmetry, it suffices to prove that if $D$ is big, then $D'$ is big. Let us denote by $D_1,\dots,D_r$ the irreducible components of $\supp(D)=\supp(D')$. We define
	\begin{center}
		$m\coloneqq \max\{m_{D_i}(D)\,|\,i=1,\dots,r\}$ and $m'\coloneqq \min\{m_{D_i}(D')\,|\, i=1,\dots,r\}$.
	\end{center}
    Let $n$ be a positive integer such that $nm'\geq m$. Then $nD'-D$ is effective. In particular, it follows that $nD'$ is big and hence so is $D'$ itself.
\end{proof}

In general, the pseudoeffective cone of a projective manifold may be very complicated to describe. However, if $X$ admits a $G$-action for some solvable linear algebraic group $G$, then we have the following very useful result concerning $\Eff(X)$.

\begin{thm}[\protect{\cite[Th{\'e}or{\`e}me~1.3]{Brion1993}}]
	\label{t.pseffconeGvariety}
	Let $G$ be a connected solvable linear algebraic group and let $X$ be a smooth projective $G$-variety. Then every effective cycle on $X$ is rationally equivalent to a $G$-stable effective cycle. In particular, the pseudoeffective cone $\Eff(X)$ of $X$ is generated by $G$-stable divisors.
\end{thm}

As an immediate application of the theorem above, one can easily derive the following criterion for bigness of $G$-equivariant vector bundles $E$ over smooth projective $G$-varieties $X$. We note that the $G$-action on $E$ induces a $G$-action on $\bP(E^*)$ such that the natural projection $\bP(E^*)\rightarrow X$ is $G$-equivariant. 

\begin{prop}[\protect{Criterion for bigness of $G$-equivariant vector bundles}]
	Let $G$ be a connected solvable linear algebraic group. Let $E$ be a $G$-equivariant vector bundle over a smooth projective $G$-variety $X$. Denote by $\pi:\bP(E^*)\rightarrow X$ the natural projection and by $\Lambda$ the tautological divisor class of $\bP(E^*)$. Then $E$ is big if and only if there exist $G$-stable effective integral divisors $\Delta$ on $\bP(E^*)$ and $D$ on $X$ satisfying:
	\begin{enumerate}
		\item there exists a positive integer $m>0$ such that $\Delta\in |m\Lambda|$;
		
		\item the divisor $D$ is big and $\Delta-\pi^*D\geq 0$.
	\end{enumerate}
\end{prop}

\begin{proof}
	One direction is clear by Lemma \ref{l.criterionbigness}. Thus we may assume that $E$ is big. By Lemma \ref{l.criterionbigness}, there exist a big divisor $A$ on $X$ and a positive integer $m_1$ such that $m_1\Lambda - \pi^*A$ is big. On the other hand, by Theorem \ref{t.pseffconeGvariety}, there exists a $G$-stable effective divisor $D'$ on $X$ such that $A\sim_{\bQ} r_1D'$ and a $G$-stable effective divisor $\Delta'$ on $\bP(E^*)$ such that $m_1\Lambda - \pi^*A \sim_{\bQ} r_2\Delta'$ for some rational number $r_1, r_2>0$. Set $\Delta=m_2(r_2\Delta'+r_1\pi^*D')$ for some sufficiently divisible positive integer $m_2$. Then we conclude by letting $m=m_1 m_2$ and $D=m_2 r_1 D'$.
\end{proof}

\subsection{Moment map of $G$-varieties}
\label{subsection:momentmap}

Let $X$ be an $n$-dimensional smooth algebraic variety. Then there exists a standard symplectic structure on the cotangent bundle $T^*X$ of $X$, which is given by a $2$-form $\omega=d\textbf{x}\wedge d\textbf{y}=\sum dx_i\wedge dy_i$, where $\textbf{x}=(x_1,\dots,x_n)$ is a tuple of local coordinates on $X$ and $\textbf{y}=(y_1,\dots,y_n)$ is an impulse, i.e., a tuple of dual coordinates in a cotangent space. If $X$ is a $G$-variety, then the symplectic structure on $T^*X$ is $G$-invariant and, for every $\xi\in \fg$, the velocity field of $\xi$ on $T^*X$ has a Hamiltonian $H_{\xi}=\xi_*$, the respective velocity field on $X$ considered as a linear function on $T^*X$. Furthermore, the action of $G$ on $T^* X$ is Hamiltonian, i.e., the map $\xi\mapsto H_{\xi}$ is a homomorphism of $\fg$ to the Poisson algebra of functions on $T^*X$. The dual morphism $\Phi_X^G:T^*X\rightarrow \fg^*$ defined as following
\[
\langle \Phi_X^G(w),\xi\rangle=H_{\xi}(w)=\langle w,H_{\xi}(x)\rangle,\ \forall w\in T_x^*X,\ \xi\in \fg,
\]
is called the \emph{moment map}. We denote by $\cM_X^G\subseteq \fg^*$ the closure of the image of the moment map. If $Z\subseteq X$ is a closed (maybe reducible) subvariety, we denote by $\cM_X^G(Z)\subseteq \fg^*$ the closure of the image $\Phi_X^G(T^*X|_Z)$. Let $T^{\fg}X\subseteq X\times \fg^*$ be the closure of the image of the following map
\[
\pi\times \Phi_X^G: T^*X \rightarrow X\times \fg^*,
\]
where $\pi:T^*X\rightarrow X$ is the natural projection. Then clearly the moment map factors as
\[
\Phi_X^G: T^*X\rightarrow T^{\fg} X \xrightarrow{\widehat{\Phi}_X^G} \fg^*.
\]
The morphism $\widehat{\Phi}_X^G$ is called the \emph{localised moment map}. The general fibres of $\widehat{\Phi}_X^G$ are the cotangent spaces $\fg_x^{\perp}= T^*_x Gx$ to general orbits and the induced map $T_x^{\fg} X \rightarrow \fg^*$ is exactly the cotangent map of the orbit map $\mu_x$ at $x$. Here $\fg_x$ is the Lie algebra of the isotropy subgroup $G_x$ of $G$ at $x$.

The moment map $\Phi_X^G$ is equivariant with respect to the natural $\bC^*$-actions on $T^*X$ and $\fg^*$. This implies that the $\Phi_X^G$ induces a \emph{projectivised moment map}
\[
\bar{\Phi}_X^G: \bP(T^*X) \dashrightarrow \bP(\fg^*).
\]
Then the closure of the image of $\bar{\Phi}_X^G$ is exactly $\bP(\cM_X^G)$ and $\cM_X^G$ is the affine cone of $\bP(\cM_X^G)$. Moreover, denote by $V\subseteq H^0(X,T_X)$ the subspace of Hamiltonians. Then $V$ can be naturally identified to a linear system $\bar{V}\subseteq |\sO_{\bP(T^*X)}(1)|$ and the rational map $\bar{\Phi}_X^G$ is exactly the rational map defined by the linear system $\bar{V}$. 

Finally, we note that the moment map $\Phi_X^G$ is $G$-equivariant and $\cM_X^G$ is a $G$-birational invariant of $X$. In particular, after passing to a smooth $G$-stable open subset, we can define the moment map for singular $G$-varieties. Moreover, the moment map also induces a homomorphism of filtered algebras
\begin{equation}
	\label{e.homo-filtered-algebras}
	\gr \Phi^*:\Sym^{\bullet}\fg=\bC[\fg^*] \rightarrow H^0(X,\Sym^{\bullet} TX)\subseteq \bC[T^*X],
\end{equation}
where $\bC[\fg^*]$ (resp. $\bC[T^*X]$) is the algebra of regular functions on $\fg^*$ (resp. $T^*X$).

If $G$ is a reductive linear algebraic group, the dimension of $\cM_X^G$ is given by Knop in \cite{Knop1990} in terms of two numerical invariants of $X$ related to the action of a Borel subgroup of $G$: the complexity and the rank.

\begin{defn}[\protect{Complexity and rank}]
	\label{d.rank-complexity}
	Let $G$ be a connected reductive linear algebraic group with a fixed Borel subgroup $B$, and let $X$ be an algebraic $G$-variety. 
	\begin{enumerate}
		\item The \emph{complexity} $c(X)$ of the action $G$ on $X$ is the codimension of a general $B$-orbit in $X$.
		
		\item The \emph{rank} $r(X)$ of the action $G$ on $X$ is the rank of $\Lambda(X)$, where $\Lambda(X)$ is the set of weights of all rational $B$-eigenfunctions on $X$.
	\end{enumerate}
\end{defn}

\begin{thm}[\protect{Dimension formula of $\cM_X^G$, \cite[Satz 7.1]{Knop1990}}]
	\label{t.KnopdimensionofMomentmap}
	Let $G$ be a connected reductive linear algebraic group and let $X$ be a projective $G$-variety. Then we have
	\[
	\dim(\cM_X^G)=\dim(\bP(\cM_X^G))+1=2\dim(X)-2c(X)-r(X).
	\]
\end{thm}

\section{Criteria for bigness and proof of Theorem \ref{t.Examples-main-thm}}

In this section, firstly we shall prove Proposition \ref{p.Intro-criterion} which gives a sufficient condition to guarantee the bigness of tangent bundles of smooth projective $G$-varieties via its moment map along effective big divisors. Next we will discuss several situations where the conditions in Proposition \ref{p.Intro-criterion} hold automatically. Finally, we apply these criteria to prove Theorem \ref{t.Examples-main-thm}, which confirms the bigness of tangent bundles of certain quasi-homogeneous spaces, including symmetric varieties, spherical varieties and equivariant compactifications of commutative linear algebraic groups.

\subsection{Criterion for bigness via moment map}
\label{subsection:criterion}

Let $m$ be a positive integer. For every $\xi \in \Sym^m \fg$, we denote by $H_{\xi}\in H^0(X,\Sym^m TX)$ the image of $\xi$ under the map $\gr\Phi^*$ \eqref{e.homo-filtered-algebras}. Then for any point $x\in X$, the evaluation $H_\xi (x)$ of $H_{\xi}$ at $x$ can be regarded as a homogeneous polynomial of degree $m$ over the cotangent space $T_x^*X$. We note that $H_\xi (x)$ may be identically zero over $T^*_x X$ and the following observation relates the vanishing locus of $H_{\xi}(x)$ in $T_x^*X$ to the zero locus of $\xi$ in $\fg^*$.

\begin{lem}
	\label{l.relationvanishinginclusion}
	Let $G$ be a connected algebraic group and let $X$ be a smooth projective $G$-variety. Given a positive integer $m$, an element $\xi\in \Sym^m \fg$ and a point $w\in T^*_x X$, then $H_\xi (x)$ vanishes at $w$ if and only if $\xi$ vanishes at $\Phi_X^G(w)\in \fg^*$.
\end{lem}

\begin{proof}
	Note that the moment map $\Phi_X^G$ restricted to $T^*_x X$ is just the following composition
	\[
	\Phi_{X,x}^G : T^*_x X \rightarrow T^*_x O_x = \fg_x^{\perp}\hookrightarrow \fg^*,
	\]
	where $O_x$ is the $G$-orbit of $x$. In particular, the form $H_{\xi}(x)$ is the composition $\xi\circ \Phi_{X,x}^G$, where $\xi$ is regarded as a function over $\fg^*$.
\end{proof}

Given a positive integer $m$ and a Weil divisor $A$ on $X$, recall that we have the following natural isomorphism
\begin{equation}
	H^0(\bP(T^*X),\sO_{\bP(T^*X)}(m)\otimes \sO_{\bP(T^*X)}(-\pi^*A)) \rightarrow H^0(X,\Sym^m T_X\otimes \sO_X(-A)).
\end{equation}
Let $\sigma\in H^0(\bP(T^*X),\sO_{\bP(T^*X)}(m))$ be a section and denote by $H_{\sigma}\in H^0(X,\Sym^m T_X)$ the corresponding symmetric vector field on $X$. Then, for a prime divisor $A$ on $X$, the section $\sigma$ vanishes along $\pi^*A$ if and only if the corresponding form $H_{\sigma}$ vanishes along $A$.

\begin{proof}[Proof of Proposition \ref{p.Intro-criterion}]
	By Lemma \ref{l.coefficient-big-divisors}, we may assume that $A$ is reduced. Since both $\cM_X^G(A)$ and $\cM_X^G$ are invariant under the dilation action of $\bC^*$ on $\fg^*$, by assumption, we have
	\[
	\dim(\bP(\cM_X^G(A)))<\dim(\bP(\cM_X^G)).
	\]
	In particular, as $\cM_X^G$ is irreducible, there is a hypersurface in $\bP(\fg^*)$ defined by a homogeneous polynomial $\xi\in \Sym^m \fg$ of degree $m$ such that it contains $\bP(\cM_X^G(A))$ but not $\bP(\cM_X^G)$. Let $H_{\xi}\in H^0(X,\Sym^m T_X)$ be the corresponding symmetric vector field on $X$. Then we have $H_{\xi}\not=0$ and $H_{\xi}$ vanishes identically along $A$. 
	
	Let $\sigma\in H^0(\bP(T^*X,\sO_{\bP(T^*X)}(m)))$ be the global section such that $H_{\sigma}=H_{\xi}$. Then according to Lemma \ref{l.relationvanishinginclusion} and the discussion before the proof, the section $\sigma$ vanishes identically along $\pi^*A$ . In particular, the following divisor
	\[
	m\Lambda - \pi^*A\sim \text{div}(\sigma)-\pi^*A\geq 0
	\]
	is pseudoeffective. Thus, as $A$ is big, it follows from Lemma \ref{l.criterionbigness} that $TX$ is big.
 \end{proof}

\begin{notation}
	\label{n.Horizontal-part}
	Given a (maybe non-reduced and reducible) hypersurface $\cH\subseteq \bP(\fg^*)$ defined by $\xi\in \Sym^m\fg$. We will denote by $D_{\xi}\in |\sO_{\bP(T^*X)}(m)|$ the divisor corresponding to $\xi$ and the divisor $D_{\cH}$ is defined as the horizontal part of $D_{\xi}$.
\end{notation}

\subsubsection{Criteria for bigness of boundary divisors}

To apply Proposition \ref{p.Intro-criterion}, firstly we need to find an effective big divisor $A$ in $X$. In most cases considered in this paper, the variety $X$ will be a smooth quasi-homogeneous projective variety and it is natural to choose $A$ to be the complement of the unique open orbit. 

\begin{prop}[\protect{\cite[\S\,I, Proposition 1 and Theorem 1]{Goodman1969}}]
	\label{p.Goodman-affinity}
	Let $X$ be a projective manifold and let $U\subseteq X$ be an affine dense open subset of $X$. Then the complement $D:=X\setminus U$ has pure codimension $1$ and the line bundle $\sO_X(D)$ is big.
\end{prop}

\begin{proof}
	The first statement follows from \cite[Proposition 1]{Goodman1969}. For the second statement, since $U$ is affine, by \cite[Theorem 1]{Goodman1969}, there is a closed subvariety $Z$ of $D$ and a blowing-up $\varphi:\bar{X}\rightarrow X$ with the centre $Z\subseteq D$ such that $\varphi^{-1}(D)$ is the support of an effective ample divisor $A$ on $\bar{X}$. In particular, the push-forward $\varphi_*A$ is a big Weil divisor with support $D$. Then it follows from Lemma \ref{l.coefficient-big-divisors} that $D$ itself is big.
\end{proof}

Let $G$ be connected linear algebraic group. A closed subgroup $H$ of $G$ is said to be \emph{regularly embedded} in $G$ if $\Rad_u(H)\subseteq \Rad_u(G)$, where $\Rad_u(H)$ (resp. $\Rad_u(G)$) is the unipotent radical of $H$ (resp. $G$).

\begin{lem}[\protect{Criteria for bigness of boundary}]
	\label{l.criteria-bigness-boundary-divisor}
	Let $G$ be a connected linear algebraic group and let $X$ be a smooth projective $G$-variety with a Zariski open dense orbit $O$. Then the complement $D\coloneqq X\setminus O$ is a big divisor if one of the following holds.
	\begin{enumerate}
		\item The group $G$ is solvable.
		
		\item For a point $x\in O$, the isotropy subgroup $G_x$ of $G$ at $x$ is regularly embedded in $G$.
		
		\item For a point $x\in O$, the isotropy subgroup $G_x$ of $G$ at $x$ is reductive.
	\end{enumerate}
\end{lem}

\begin{proof}
	By Proposition \ref{p.Goodman-affinity}, it is enough to show that $O$ is an affine variety and the latter follows from certain known criteria for affineness of homogeneous spaces (see for instance \cite[Theorem 3.5, Theorem 3.7 and Theorem 3.8]{Timashev2011}).
\end{proof}

\subsubsection{Image of moment map along boundary divisors}

Once we have an effective big divisor $D$ on a projective $G$-variety $X$, to apply Proposition \ref{p.Intro-criterion}, we need to control the dimension of $\cM_X^G(D_{\red})$. In the following we consider the case where $D$ is $G$-stable.

\begin{lem}
	\label{l.orbit-normaliser}
	Let $G$ be a connected algebraic group and let $X=G/H$ be a homogeneous variety. Denote by $N$ the normaliser $N_G(H)$ of $H$ in $G$. Then we have $T^*X=G\ast_H \fh^{\perp}$ and the moment map $\Phi_X^G$ factors as
	\[
	G\ast_H \fh^{\perp} \rightarrow G\ast_N \fh^{\perp} \rightarrow \fg^*,
	\]
	where $N$ acts on $\fh^{\perp}$ by coadjoint action. In particular, we have 
	\begin{equation}
		\label{e.dimension-normaliser}
		\dim(\cM_X^G)\leq \dim(X)+\dim(\fg)-\dim(\fn).
	\end{equation}
\end{lem}

\begin{proof}
	This is a standard fact about homogeneous spaces. Let us recall the proof for the reader's convenience. There is a canonical isomorphism $T_x X\cong \fg/\fg_x$ for any point $x\in X$. In particular, we get a canonical embedding $T^*X\rightarrow X\times \fg^*$. Then the isomorphism $\Psi:G\ast_H {\fh}^{\perp}\rightarrow T^*X$ is given as follows:
	\[
	\Psi(g,w)=([gH],\text{Ad}_g^*(w)),
	\]
	where $g\in G$, $w\in \fh^{\perp}$ and $\text{Ad}_g^*\in \GL(\fg^*)$ is the coadjoint representation of $\fg$. In particular, under the isomorphism $\Psi$, the moment map $\Phi_X^G:T^*X\rightarrow \fg^*$ can be written as 
	\[
	\Phi_X^G(g\ast w)=\text{Ad}_g^*(w),
	\]
	where $g\in G$ and $w\in \fh^{\perp}$. For any $n\in N$, we have $\text{Ad}_n^*(\fh^{\perp})=\fh^{\perp}$ by definition. This immediately implies that $\Phi_X^G$ factors through $G\ast_N \fh^{\perp} \rightarrow \fg^*$. The inequality \eqref{e.dimension-normaliser} then follows from the fact that $\dim(\cM_X^G)\leq \dim(G\ast_N \fh^{\perp})=\dim(G/N)+\dim(\fh^{\perp})$.
\end{proof}

\begin{lem}
	\label{l.normaliser-divisor}
	Let $G$ be a connected linear algebraic group and let $X$ be a smooth projective $G$-variety. Let $D$ be a $G$-stable prime divisor in $X$. Then we have
	\[
	\dim(\cM_X^G(D))=\dim(\cM_D^G)\leq \dim(X) + \dim(\fg) - \dim(\fn) - 1,
	\]
	where $\fn$ is the Lie algebra of the normaliser $N_G(H)$ of the isotropy subgroup of $G$ at a general point $x\in D$. 
\end{lem}

\begin{proof}
	Since $D$ is $G$-stable, the restriction of the moment map $\Phi_X^G$ to $D$ factors as
	\[
	 T^*X|_{D_{\reg}} \rightarrow T^*D_{\reg} \rightarrow \fg^*,
	\]
	where $D_{\reg}$ is the smooth locus of $D$. In particular, we have $\cM_X^G(D)=\cM_D^G$. Let $O_x$ be the orbit of a general point $x\in D_{\reg}$. Then we also have
	\[
	\dim(\cM_D^G)=\dim(\widehat{\Phi}_D^G(T^{\fg} D_{\reg}))=\dim(D)-\dim(O_x)+\dim(\cM_{O_x}^G).
	\]
	Then we conclude by applying Lemma \ref{l.orbit-normaliser} to the homogeneous space $O_x=G/G_x$.
\end{proof}

\begin{lem}
	\label{l.Gstable}
	Let $G$ be a connected reductive linear algebraic group and let $X$ be a smooth projective $G$-variety. Let $D$ be a $G$-stable prime divisor in $X$. Then
	\[
	\dim(\cM_X^G(D))=\dim(\cM_D^G)<\dim(\cM_X^G).
	\]
	if and only if
	\begin{center}
		$c(X)=c(D)$ and $r(X)=r(D)+1$.
	\end{center}
\end{lem}

\begin{proof}
	This follows directly from Knop's dimension formula, see Theorem \ref{t.KnopdimensionofMomentmap}. Here we remark that we have always $c(D)\leq c(X)$, $r(D)\leq r(X)$ and the equality holds if and only if $D=X$ \cite[Theorem 5.7]{Timashev2011}.
\end{proof}

\begin{example}[\protect{Quintic del Pezzo threefold}]
	\label{e.V5}
	Let $X=V_3$ be the smooth quintic del Pezzo threefold, e.g., a smooth codimension $3$ linear section of $\Gr(2,5)\subseteq \bP^9$. Then $TX$ is big by \cite[Theorem 1.5]{HoeringLiuShao2020}. Denote by $H$ the ample generator of $\Pic(X)$. Recall that there is an $\SL_2$-action on $X$ with three orbits \cite{MukaiUmemura1983}: a unique open orbit, a unique $2$-dimensional orbit whose closure $D$ is linearly equivalent to $2H$, a unique $1$-dimensional orbit which is a rational normal curve of degree $6$. Then under the $\SL_2$-action,  we have 
	\begin{center}
		$c(X)=r(X)=1$, $c(D)=0$ and $r(D)=1$.
	\end{center}
	Hence, by Theorem \ref{t.KnopdimensionofMomentmap}, we obtain that $\cM_X^{G}=\cM_D^G=\fg^*$. This shows that the converse of Proposition \ref{p.Intro-criterion} is false in general.
\end{example}

\subsection{Proof of Theorem \ref{t.Examples-main-thm}}

\label{subsection.proofmainthm}

In this subsection, we aim to apply the criteria proved in the previous subsection to prove the bigness of the tangent bundles of certain interesting quasi-homogeneous $G$-varieties. In particular, we shall finish the proof of Theorem \ref{t.Examples-main-thm}.

\subsubsection{Spherical varieties}

Let $G$ be a connected reductive linear algebraic group and let $X$ be a normal $G$-variety. Then $X$ is said to be \emph{spherical} if $c(X)=0$. In particular, there is an open $G$-orbit $G/H\subseteq X$. Let $Y$ be a $G$-orbit in $X$. Denote by $\cV_Y(X)$ the set of $G$-stable prime divisors in $X$ containing $Y$ and by $\cD_Y(X)$ the set of $B$-stable but not $G$-stable prime divisors in $X$ containing $Y$. Write $\bX_{\bQ}^\vee$ the tensor product of the dual lattice of $X$ with $\bQ$. For any prime divisor $D$ in $X$ there is an associated valuation $\nu_D$ and also an associated element $\rho(D)$ in $\bX_{\bQ}^{\vee}$. Denote by $\cC_Y^{\vee}(X)\subseteq \bX_{\bQ}^{\vee}$ the cone generated by the images of divisors in $\cV_Y(X)$ and $\cD_Y(X)$.

\begin{example}
	\label{e.Examples-Spherical-Varieties}
	We collect some typical examples of spherical varieties.
	\begin{enumerate}
		\item Recall that a toric variety is a normal variety $X$ with a dense orbit of a torus $T=\bG_m^r$ such that the points in the dense $T$-orbit have a trivial stabiliser in $T$. The variety $X$ is spherical for $G=T$ with $H=\{e\}$. Moreover, we have $r(X)=\dim(T)=\dim(X)$. Recall that it is shown in \cite[Corollary 1.3]{Hsiao2015} that the tangent bundle of a smooth projective toric variety is big.
		
		\item For $G$ a connected semi-simple linear algebraic group and $P$ a parabolic subgroup containing a maximal torus $T$ of $G$, the quotient $G/P$ is a projective rational homogeneous space and the Bruhat decomposition implies that $G/P$ is a spherical $G$-variety. Moreover, we have $r(X)=0$. Thus the moment map $\bar{\Phi}_X^G$ is generically finite \cite{Richardson1974} and hence $TX$ is big (cf. Theorem \ref{t.KnopdimensionofMomentmap}).
		
		\item Let $G$ be a connected semi-simple linear algebraic group equipped with a non-identical involution $\theta\in \Aut(G)$. Let $H$ be a closed subgroup of $G$ such that $G^{\theta}\subseteq H\subseteq N_G(G^{\theta})$. Then $G/H$ is said to be a \emph{symmetric homogeneous space} and $G/H$-embeddings are called \emph{symmetric varieties}. The symmetric varieties are spherical.
	\end{enumerate} 
\end{example}

\begin{prop}
	\label{p.G-spherical}
	Let $G$ be a connected reductive linear algebraic group and let $X$ be smooth projective spherical $G$-variety. Let $D$ be a $G$-stable prime divisor in $X$. Then we have $c(X)=c(D)=0$ and $r(D)=r(X)-1$. In particular, if there exists a $G$-stable affine open subset $O$ of $X$, then $TX$ is big.
\end{prop}

\begin{proof}
	Since $X$ is a spherical $G$-variety, we have $c(D)\leq c(X)=0$. On the other hand, let $Y$ be the unique open $G$-orbit in $D$. Then one can easily obtain that $\cV_{Y}(X)=\{D\}$ and $\cD_{Y}(X)=\emptyset$. This implies that the cone $\cC_{Y}^{\vee}(X)\subseteq \bX_{\bQ}^{\vee}$ is $1$-dimensional. In particular, by \cite[Proposition 15.14]{Timashev2011}, we obtain 
	\[
	r(D)=r(Y)=r(X)-\dim(\cC_{Y}^{\vee})=r(X)-1.
	\]
	Now we assume that there exists a $G$-stable affine open subset $O$ of $X$. By Proposition \ref{p.Goodman-affinity}, the complement $D\coloneqq X\setminus O$ is a big divisor on $X$. Then the bigness of $TX$ follows from Proposition \ref{p.Intro-criterion} and Lemma \ref{l.Gstable}.
\end{proof}

\begin{cor}
	\label{c.symmetric-Picard-number-one}
	Let $G$ be a connected reductive linear algebraic group and let $X$ be a smooth projective spherical $G$-variety. Then $TX$ is big if one of the following holds.
	\begin{enumerate}
		\item The variety $X$ is a symmetric variety.
		
		\item The variety $X$ has Picard number $1$ and contains a $G$-stable prime divisor.
	\end{enumerate}
\end{cor}

\begin{proof}
	Firstly we assume that $X$ is a symmetric variety and denote by $O$ the unique open $G$-orbit of $X$. Then $O$ is isomorphic to a symmetric homogeneous variety $G/H$. On the other hand, thanks to \cite[\S\,8]{Steinberg1968}, the subgroup $H$ is reductive. Therefore, by Lemma \ref{l.criteria-bigness-boundary-divisor}, the boundary divisor $D\coloneqq X\setminus O$ is a big divisor and it follows from Proposition \ref{p.G-spherical} above that $TX$ is big.
	
	Next we assume that $X$ has Picard number $1$ and there exists a $G$-stable prime divisor $D$ in $X$. Then $D$ is ample and it is well known that $O\coloneqq X\setminus D$ is an affine variety. Hence it follows again from Proposition \ref{p.G-spherical} that $TX$ is big.
\end{proof}

\begin{rem}
	\label{r.Classification-Symmetric-varieties}
	The smooth projective symmetric varieties with Picard number $1$ are classified by Ruzzi in \cite{Ruzzi2010} and there are exactly six non-homogeneous ones, including the Cayley Grassmannian $CG$, the double Cayley Grassmannian $DG$, a general hyperplane section of $\Gr_{\omega}(\bA^3,\bA^6)$, where $\bA$ is a complex composition algebra. In particular, by Semicontinuity Theorem, the tangent bundle of any smooth hyperplane section of $\Gr_{\omega}(\bA^3,\bA^6)$ is big. 
\end{rem}

\subsubsection{Horospherical varieties}

Let $G$ be a connected reductive linear algebraic group. A closed subgroup $H$ of $G$ is said to be \emph{horospherical} if it contains the unipotent radical of a Borel subgroup of $G$. In this case we shall say that the homogeneous space $G/H$ is \emph{horospherical}. Denote by $P$ the normaliser $N_G(H)$ of $H$ in $G$. Then $P$ is a parabolic subgroup of $G$ such that $P/H$ is a torus and $G/H$ is a torus bundle over the flag variety $G/P$. A normal $G$-variety is said to be \emph{a horospherical variety} if $G$ has an open orbit isomorphic to $G/H$ for some horospherical subgroup $H$. Horospherical varieties are spherical and their ranks are equal to the rank of the torus $P/H$.

\begin{prop}
	\label{p.horospherical}
	Let $G$ be a connected reductive linear algebraic group and let $X$ be a smooth projective horospherical $G$-variety. Then $TX$ is big.
\end{prop}

\begin{proof}
	Let $D$ be a $G$-stable prime divisor in $X$. As shown in the proof of Proposition \ref{p.G-spherical}, we have $r(D)=r(X)-1$ and, by Lemma \ref{l.Gstable}, we obtain $\dim(\cM_X^G(D))<\dim(\cM_X^G)$.
	Let $O=G/H$ be the unique open $G$-orbit in $X$ with $H$ a horospherical subgroup of $G$. Denote by $P=N_G(H)$ the normaliser of $H$ in $G$. Then, by Lemma \ref{l.orbit-normaliser}, the restriction of the moment map $\Phi_X^G:T^*X\rightarrow \fg^*$ to $O$ factors as
	\[
	\begin{tikzcd}[column sep=large]
		G \ast_H \fh^{\perp} \ar[r,"{\pi_A}"]
		    & G\ast_P \fh^{\perp} \ar[r,"{\varphi}"]
		        & \fg^*.
	\end{tikzcd}
	\]
	Moreover, it is known that $\varphi$ is generically finite onto its image $\cM_X^G$ (cf. Theorem \ref{t.KnopdimensionofMomentmap}). Let $D$ be a $B$-stable but not $G$-stable prime divisor in $X$. Then $D\cap O\not=\emptyset$ and $D\cap O$ is the inverse image by the torus fibration $G/H\rightarrow G/P$ of a Schubert divisor $D'$ of the flag variety $G/P$. As a consequence, the Zariski closure of the image $\Phi_X^G(T^*X|_D)$ is equal to the Zariski closure of the image $\varphi(p^{-1}(D'))$, where $p:G\ast_P \fh^{\perp}\rightarrow G/P$ is the natural projection. However, as $\varphi$ is generically finite onto $\cM_X^G$, we get
	\[
	\dim(\varphi(p^{-1}(D')))<\dim(G\ast_P \fh^{\perp})=\dim(\cM_X^G).
	\]
	As a consequence, our argument above shows that for every $B$-stable prime divisor $D$ in $X$, we have always $\dim(\cM_X^G(D))<\dim(\cM_X^G)$. On the other hand, let $O_B$ be the open $B$-orbit of $X$. Then $O_B$ is an affine variety and the complement $D\coloneqq X\setminus O_B$ is big divisor by Lemma \ref{l.criteria-bigness-boundary-divisor}. Then it follows from Proposition \ref{p.Intro-criterion} that $TX$ is big.
\end{proof}

\begin{rem}
	\label{r.Examples-Horospherical}
	Smooth projective horospherical varieties with Picard number $1$ are classified by Pasquier in \cite{Pasquier2009}. With the same notations as in \cite{Pasquier2009}, there are five classes of non-homogeneous ones, including $X^1(m)=(B_m,\omega_{m-1},\omega_m)$ $(m\geq 3)$, $X^2=(B_3,\omega_1,\omega_3)$, $X^3(m,i)=(C_m,\omega_i,\omega_{i+1})$ $(m\geq 2, 1\leq i\leq m-1)$, $X^4=(F_4,\omega_2,\omega_3)$ and $X^5=(G_2,\omega_2,\omega_1)$.
\end{rem}

\subsubsection{Quasi-homogeneous $G$-varieties with $G$ commutative}

The following result confirms the bigness of the tangent bundles of equivariant compactifications of connected commutative linear algebraic groups.

\begin{prop}
	\label{p.commutative-groups}
	Let $G$ be a connected commutative linear algebraic group and let $X$ be a smooth projective $G$-variety with an open $G$-orbit $O$. Then $TX$ is big. In particular, the tangent bundle of an equivariant compactification of $G$ is big.
\end{prop}

\begin{proof}
	Let $O$ be the unique open $G$-orbit in $X$ and let $D\coloneqq X\setminus O$ be the complement of $O$. Since $G$ is solvable, by Lemma \ref{l.criteria-bigness-boundary-divisor}, the divisor $D$ is big. Moreover, as $G$ is commutative, for any subgroup $H$, we have always $N_G(H)=G$. In particular, Lemma \ref{l.orbit-normaliser} implies that $\dim(\cM_X^G)\leq \dim(X)$. On the other hand, as $G$ has an open orbit $O$ in $X$, we must have $\dim(\cM_X^G)\geq \dim(O)=\dim(X)$. Hence, we obtain $\dim(\cM_X^G)=\dim(X)$. Let $D_i$ be an irreducible component of $D$. As $G$ is commutative, Lemma \ref{l.normaliser-divisor} yields
	\[
	\dim(\cM_X^G(D_i))=\dim(\cM_{D_i}^G)\leq \dim(X)-1<\dim(\cM_X^G).
	\]
	Hence, the tangent bundle $TX$ is big by Proposition \ref{p.Intro-criterion}.
\end{proof}

\begin{proof}[Proof of Theorem \ref{t.Examples-main-thm}]
	It follows from Proposition \ref{p.G-spherical}, Corollary \ref{c.symmetric-Picard-number-one}, Proposition \ref{p.horospherical} and Proposition \ref{p.commutative-groups}.
\end{proof}

\begin{rem}
	Recall that a connected commutative linear algebraic group is known to be isomorphic to $\bG_m^r\times \bG_a^s$ with some non-negative integers $r$ and $s$. 
	\begin{enumerate}
		\item If $s=0$, then $G=\bG_m^r$ is a torus and an equivariant compactification of $G$ is a toric variety. In particular, our result above recovers the bigness of tangent bundles of smooth projective toric varieties \cite[Corollary 1.3]{Hsiao2015}.
		
		\item If $r=0$, then $G=\bG_a^s$ is a vector group and the equivariant compactifications of vector groups are studied actively during the past decades. A full classification of all Fano threefolds admitting an equivariant compactification structure of the vector group $\bG_a^3$ is given in \cite[Main Theorem]{HuangMontero2020}, including $14$ toric ones and $5$ non-toric ones. In particular, this allows us to give a different proof of the bigness of the tangent bundles of the Fano threefolds \textnumero~$28$, \textnumero~$30$ and \textnumero~$31$ in \cite[Table 1]{KimKimLee2022}, which are proved there using total dual VMRT. In higher dimension, a classification of Fano manifolds admitting an equivariant compactification structure of vector groups is available only for varieties with high index (see \cite{FuMontero2019}, Example \ref{e.Examples-EC-Picard-number-one} and Example \ref{e.Examples-EC-structure}). The equivariant compactifications of vector groups with Picard number $1$ are of special interests and we will discuss them in details in the next section.
	\end{enumerate}
\end{rem}

\section{Equivariant compactification of vector groups}

In this section, we will investigate Fano manifolds with Picard number $1$ which is an equivariant compactification of a vector group $\bG_a^n$. The study of equivariant compactification of vector groups is started in \cite{HassettTschinkel1999}, where a classification of them in dimension $3$ and of Picard number $1$ is obtained. Nevertheless, it seems difficult to obtain a full classification in higher dimension, see \cite{FuHwang2014,FuMontero2019,FuHwang2020} for more details. The main goal of this section is to show that the image of the projectivised moment map $\bar{\Phi}_X^G$ along the boundary divisor of an equivariant compactification of a vector group is projectively equivalent to the dual variety of its VMRT. In particular, this allows us to relate the criterion for the bigness of tangent bundles given in Proposition \ref{p.Intro-criterion} via moment map to the previous approach to the bigness of tangent bundles via total dual VMRT initiated in \cite{HoeringLiuShao2020} (see also \cite{FuLiu2021,HoeringLiu2021,KimKimLee2022} and Theorem \ref{t.bigness-total-dual-VMRT} below).

\subsection{VMRT and its dual variety}

\label{subsection.VMRTdual}

Let $X$ be a uniruled projective manifold. An irreducible component $\cK$ of the space of rational curves on $X$ is called a \emph{minimal rational component} if the subscheme $\cK_x$ of $\cK$ parametrising curves passing through a general point $x\in X$ is non-empty and proper. Curves parametrised by $\cK$ will be called \emph{minimal rational curves}. Let $q:\cU\rightarrow \cK$ be the universal family and by $\mu:\cU\rightarrow X$ the evaluation map. The tangent map $\tau:\cU\dashrightarrow \bP(TX)$ is defined by 
\[
\tau(u)=[T_{\mu(u)}\mu(q^{-1}q(u))]\in \bP(T_{\mu(x)}X).
\]
The closure $\cC\subseteq \bP(TX)$ of its image is the \emph{total variety of minimal rational tangents} (total VMRT for short) of $X$. The projection $\cC\rightarrow X$ is a proper and surjective morphism, and a general fibre $\cC_x\subseteq \bP(T_x X)$ is called the \emph{variety of minimal rational tangents} (VMRT for short) of $X$ at the point $x\in X$.  A general minimal rational curve $l$ passing through a general point $x$ is standard; that is, if $f:\bP^1\rightarrow l$ is the normalisation, we have
\[
f^*TX \cong \sO_{\bP^1}(2)\oplus \sO_{\bP^1} (1)^{\oplus p} \oplus \sO_{\bP^1}^{\oplus (n-p-1)},
\]
where $p=\dim(\cC_x)$. Moreover, the projectivised tangent space $\textbf{T}_{[T_x l]} \cC_x$ of $\cC_x$ at $[T_x l]$ is the linear subspace of $\bP(T_x X)$ corresponding to the positive factors of $f^*TX$ at $x\in l$.

Let $Z\subseteq \bP(V)$ be a projective variety. The \emph{dual variety} $\check{Z}\subseteq \bP(V^*)$ is defined as the Zariski closure of the set of hyperplanes in $\bP(V)$ which are tangent to $Z$ at some smooth point. The \emph{dual defect} of $Z$ is $\codim(\check{Z})-1$ and $Z$ is called \emph{dual defective} if its dual variety $\check{Z}$ is not a hypersurface. The \emph{codegree} of $Z$ is defined as the degree of $\check{Z}$. The \emph{total dual VMRT} $\check{\cC}\subseteq \bP(T^*_x X)$ of a uniruled projective manifold $X$ is defined as the closure of the union of dual varieties $\check{\cC}_x\subseteq \bP(T^*_x X)$ of the VMRT $\cC_x\subseteq \bP(T_x X)$ at general points $x\in X$. The importance of the total dual VMRT in the study of the bigness of tangent bundles is illustrated in the following theorem.

\begin{thm}[\protect{\cite[Theorem 3.4]{FuLiu2021} and \cite[Proposition 5.8]{HoeringLiu2021}}]
	\label{t.bigness-total-dual-VMRT}
	Let $X$ be a Fano manifold with Picard number $1$ and denote by $H$ the ample generator of $\Pic(X)$. Assume that the VMRT of $X$ at a general point is not dual defective and denote by $a\in \bZ_{>0}$, $b\in\bZ$ the unique integers such that 
	\[
	[\check{\cC}]\equiv a\Lambda + b\pi^*H,
	\]
	where $\Lambda$ is the tautological divisor class of $\bP(T^* X)$ and $\pi:\bP(T^*X)\rightarrow X$ is the natural projection. Then $TX$ is big if and only if $b<0$.
\end{thm}

The following result suggests that there may exist some interesting relations between the criterion given in Proposition \ref{p.Intro-criterion} via moment map and that given in Theorem \ref{t.bigness-total-dual-VMRT} via total dual VMRT.

\begin{lem}
	\label{l.inclusion-dual-VMRT}
	Let $G$ be a connected algebraic group and let $X$ be a smooth projective uniruled $G$-variety. Fix a minimal rational component $\cK$ on $X$ with total dual VMRT $\check{\cC}\subseteq \bP(T^*X)$. Then for any reduced big divisor $D$ in $X$ we have
	\[
	\bar{\Phi}_X^G(\check{\cC})\subseteq \bar{\Phi}_X^G(\bP(T^*X|_D))=\bP(\cM_X^G(D)).
	\]
\end{lem}

\begin{proof}
	The inclusion is clear if $\cM_X^G(D)=\cM_X^G$. Thus we may assume that $\cM_X^G(D)\not=\cM_X^G$. Let $\cH\subseteq \bP(\fg^*)$ be an arbitrary reduced (maybe reducible) hypersurface of degree $m$ containing $\bP(\cM_X^G(D))$, but not containing $\bP(\cM_X^G)$. Consider the divisor $D_{\cH}\subseteq \bP(T^*X)$. Then there exists an effective big divisor $D'\geq D$ such that $D_{\cH}+\pi^*D'\sim m\Lambda$, where $\Lambda$ is the tautological divisor of $\pi:\bP(T^*X)\rightarrow X$. On the other hand, note that $\check{\cC}$ is dominated by minimal sections $\bar{C}$ over standard rational curves in $\cK$ (cf. \cite[Proposition 2.9]{HoeringLiuShao2020}), which have $\Lambda$-degree $0$. This implies that the restriction $D_{\cH}|_{\check{\cC}}$ is not pseudoeffective since $D'$ is big and $\pi^*D'\cdot \bar{C}>0$. In particular, the total dual VMRT $\check{\cC}$ is contained in the support of the divisor $D_{\cH}$. By the construction of $D_{\cH}$, we must have $\bar{\Phi}_X^G(D_{\cH})\subseteq \cH$ and therefore $\bar{\Phi}_X^G(\check{\cC})\subseteq \cH$. As $\cH$ is an arbitrary reduced hypersurface containing $\bP(\cM_X^G(D))$, it follows that $\bar{\Phi}_X^G(\check{\cC})\subseteq \bP(\cM_X^G(D))$.
\end{proof}

\begin{rem}
	The assumption on the bigness of $D$ can not be removed, see Example \ref{e.Examples-EC-structure} below. Moreover, the following inclusion in general is strict:
	\[
	\bar{\Phi}_X^G(\check{\cC}) \subseteq \bigcap_{D\ \text{effective big divisor}} \bP(\cM_X^G(D)).
	\]
	Let $X=G/P$ be a rational homogeneous space of Picard number $1$. Then the moment map $\Phi_X^G:T^*X\rightarrow \fg^*$ is a generically finite dominant map to its image $\cM_X^G$, which is the closure of a nilpotent orbit. Moreover, the projectivised moment map $\bar{\Phi}_X^G:\bP(T^*X)\rightarrow \bP(\cM_X^G)$ is everywhere well-defined. Let $\bP(T^*X)\xrightarrow{\varepsilon} \bP(\widetilde{\cM}_X^G) \rightarrow \bP(\cM_X^G)$ be the Stein factorisation. Then the total dual VMRT $\check{\cC}$ is contained in the exceptional locus $E$ of $\varepsilon$ and one can obtain (see \cite[\S\,5]{FuLiu2021} for more details):
	\[
	 \bigcap_{D\ \text{effective ample divisor}} \bP(\cM_X^G(D)) = \bar{\Phi}_X^G(E).
	\]
	Nevertheless, if $E$ is a divisor and $\check{\cC}$ is not a divisor, then $\bar{\Phi}_X^G(\check{\cC})$ is a proper subvariety of $\bar{\Phi}_X^G(E)$, see \cite[Definition 5.6 and Table 2]{FuLiu2021}.
\end{rem}

\subsection{Geometry of equivariant compactifications}

In this subsection, we collect some basic facts about equivariant compactifications of vector groups. Recall that for a smooth projective variety $X$, an \emph{EC-structure} on $X$ is an algebraic action $\bG_a^n\times X\rightarrow X$ which makes $X$ an equivariant compactification of $\bG_a^n$.

\begin{prop}[\protect{\cite[Proposition 5.4]{FuHwang2020}}]
	\label{p.factsofEC}
	Let $X$ be a Fano manifold of Picard number $1$ which is an equivariant compactification of $\bG_a^n$. Denote by $D$ the complement of the unique open $\bG_a^n$-orbit $O\subseteq X$. Let $\cK$ be a covering family of minimal rational curves on $X$ and denote by $\cC\subseteq \bP(TX)$ its total VMRT. Then the following statements hold.
	\begin{enumerate}
		\item The closed subvariety $D$ is an irreducible divisor such that $\Pic(X)\cong \bZ D$.
		
		\item If the points in $D$ are fixed by $\bG_a^n$, then $X$ is isomorphic to $\bP^n$.
		
		\item For any point $x\in O$, the VMRT $\cC_x\subseteq \bP(T_x X)$ is irreducible and is independent of $x$ up to projective equivalence.
		
		\item If the VMRT is smooth, then a member $C$ of $\cK_x$, for $x\in O$, is the closure of the image of a $1$-dimensional subspace in $\bG_a^n$ and $D\cdot C=1$.
	\end{enumerate}
\end{prop}

\begin{proof}
	The first statement is proved in \cite[Theorem 2.5]{HassettTschinkel1999}. For the second statement, for every $\xi \in \fg$, the velocity field $H_{\xi}\in H^0(X,TX)$ vanishes identically along $D$ by  Lemma  \ref{l.relationvanishinginclusion}. In particular, we have $H^0(X,TX\otimes \sO_X(-D))\not=\emptyset$. As $D$ is ample, it follows from \cite{Wahl1983} (see 
	also \cite{Liu2019}) that $X$ is isomorphic to $\bP^n$.	The first part of the third statement is proved in \cite[Proposition 2.2]{FuHwang2014} and the second part follows from the fact that the total VMRT $\cC\subseteq \bP(TX)$ is preserved by the natural action of $\bG_a^n$ on $\bP(TX)$.  The last statement follows from \cite[Proposition 5.4]{FuHwang2020}.
\end{proof}

Let $Z\subseteq \bP(V)$ be a nondegenerate submanifold and let $W\subseteq V$ be a subspace such that $\bP(W)\subseteq Z$. Denote by $(V/W)^*\subset V^*$ the set of linear functionals on $V$ annihilating $W$ such that $\bP((V/W)^*)$ parametrises the set of hyperplanes in $\bP(V)$ containing $\bP(W)$. Then a general member of $\bP((V/W)^*)$ is called a \emph{$\bP(W)$-general} hyperplane in $\bP(V)$. More generally, a linear subspace of codimension $k$ in $\bP(V)$ is \emph{$\bP(W)$-general} if it is defined by a general member of $\Gr(k,(V/W)^*)$.

\begin{example}
	\label{e.Examples-EC-Picard-number-one}
	Up to our knowledge, the known examples of Fano manifolds with Picard number $1$ which are equvariant compactifications of vector groups are as follows: 
	\begin{itemize}
		\item the irreducible Hermitian symmetric spaces;
		\item the odd Lagrangian Grassmannians $X^3(m,m-1)$ $(m\geq 2)$ (cf. Remark \ref{r.Examples-Horospherical});
		\item a smooth linear section $V_k$ of $\Gr(2,5)\subseteq \bP^9$ with codimension $k\leq 2$;
		\item a smooth $\bP^4$-general linear section $S^a_k$ of $\bS_5\subseteq \bP^{15}$ with codimension $k\leq 3$.
	\end{itemize}
\end{example}

The following example shows that there are many smooth equivariant compactifications of vector groups with higher Picard number. 
\begin{example}[\protect{\cite[Example 2.2]{FuHwang2020}}]
	\label{e.Examples-EC-structure}
	Let $[x_0:\dots:x_{n}]$ be the homogeneous coordinates of the $n$-dimensional projective space $\bP^n$. Let $H=\bP^{n-1}\subseteq \bP^n$ be the hyperplane defined by the equation $x_0=0$. Then there is a natural EC-structure $\Psi:\bG_a^n\times \bP^n\rightarrow \bP^n$ on $\bP^n$ with the unique open orbit $\bP^n\setminus H$. More precisely, for a point $\textbf{y}=(y_1,\dots,y_n)\in \bG_a^n$, we define an automorphism $\Psi_{\textbf{y}}:\bP^n\rightarrow \bP^n$ as follows:
	\[
	[x_0:x_1:\dots:x_n] \mapsto [x_0:x_1+y_1 x_0:\dots:x_n+y_n x_0].
	\]
	Clearly this gives a EC-structure on $\bP^n$ such that the induced action on the hyperplane $H$ is trivial. Let $S\subseteq H$ be a smooth irreducible projective variety, and let 
	\[
	\nu:Z\coloneqq\text{Bl}_S\bP^n\rightarrow \bP^n
	\]
	be the blowing-up of $\bP^n$ along $S$ with exceptional divisor $E=\bP(N_{S/\bP^n})$. Then the EC-structure $\Psi$ on $\bP^n$ can be naturally lifted to be EC-structure $\Psi_Z$ on $Z$ such that $\mu$ is equivariant. Denote by $W\subseteq \bP(N_{S/\bP^n})$ the subvariety $\bP(N_{S/H})$ of $E$. Then it is clear that the induced action of $\Phi_Z$ on $W$ is trivial and for each point $s\in S$, the fibre $E_s$ of $E\rightarrow S$ over $s$ is invariant such that $\Phi_Z$ is transitive over the open subset $E_s\setminus W_s$, where $W_s$ is the fibre of $W\rightarrow S$ over $s$.
	
	Denote by $\widetilde{H}$ the strict transform of $H$ in $Z$. Then the induced $\bG_a^n$-action on $\widetilde{H}$ is trivial. In particular, the image $\Phi_Z^{\bG_a^n}(T^*Z|_{\widetilde{H}})$ is the origin $0\in \fg^*$. Let $\cK$ be the irreducible component of the space of rational curves in $Z$ parametrising the strict transforms of lines in $\bP^n$ meeting $S$. Then $\cK$ is a minimal rational component on $Z$ such that its VMRT is projectively equivalent to $S\subseteq \bP^{n-1}$. For any point $z\in Z\setminus(\widetilde{H}\cup E)$, we have $\Phi_Z^{\bG_a^n}(T^*_z Z)=\fg^*$. Since the members in $\cK$ have $\widetilde{H}$-degree $0$, the divisor $\widetilde{H}$ is not big. This shows that the assumption on the bigness of $D$ in Lemma \ref{l.inclusion-dual-VMRT} can not be removed.
\end{example}

The proof of the following result is communicated to me by Baohua Fu.

\begin{prop}
	\label{p.boundary-divisor-EC}
	Let $X$ be a Fano manifold of Picard number $1$ which is an equivariant compactification of $\bG_a^n$, different from projective spaces. Denote by $D\subseteq X$ the boundary divisor. Assume that there exists a covering family $\cK$ of minimal rational curves on $X$ such that its VMRT $\cC_x\subseteq \bP(T_x X)$ at a general point $x\in X$ is projectively equivalent to a smooth projective variety $S\subseteq \bP^{n-1}$. Let $\pi:\Bl_S(\bP^n)\rightarrow \bP^n$ be the blow-up of $S\subseteq \bP^{n-1}=H\subseteq \bP^{n}$. Denote by $\widetilde{H}$ the strict transform of $H$ and by $E$ the exceptional divisor of $\pi$. Then there exist $\bG_a^n$-stable proper subvarieties $D_0\subseteq D$ and $E_0\subseteq E$ such that there exists a $\bG_a^n$-equivariant isomorphism 
	\[
	\Phi:X\setminus D_0\rightarrow \Bl_S(\bP^n)\setminus (\widetilde{H}\cup E_0).
	\]
\end{prop}

\begin{proof}
	Let $\cK'$ be the covering family of minimal rational curves on $\Bl_S(\bP^n)$ parametrising the strict transform of lines in $\bP^n$ meeting $S$. Denote by $O$ and $O'$ the open $\bG_a^n$-orbits of $X$ and $\Bl_S(\bP^n)$, respectively. Fix two points $o\in O$ and $o'\in O'$. Denote by $\nu_o:\bG_a^n\rightarrow O$ and $\nu_{o'}:\bG_a^n\rightarrow O'$ the orbit maps, respectively. For a general point $x'\in \Bl_S(\bP^n)$, we denote by $\cC'_{x'}\subseteq \bP(T_{x'}\Bl_S(\bP^n))$ the VMRT of $\cK'$ at $x'$. Note that the VMRT of $\cK$ at a point $x\in O'$ is projectively equivalent to $S\subseteq \bP^{n-1}$. In particular, applying \cite[Proposition 2.4]{FuHwang2014} shows that there exists a group automorphism $F$ of $\bG_a^n$ such that the biholomorphic map $\varphi:O\rightarrow O'$ defined by $\varphi\coloneqq \nu_{o'}\circ F \circ \nu_o^{-1}$ satisfying:
	\begin{enumerate}
		\item $\varphi(o)=o'$;
		
		\item \label{s.Eequivariantness} $\varphi(g\cdot o)=F(g)\cdot \varphi(o')$ for any $g\in \bG_a^n$; and
		
		\item \label{s.preseveringVMRT} the differential map $d\varphi:\bP(TO) \rightarrow \bP(TO')$ sends $\cC_x$ to $\cC'_{\varphi(x)}$ for all $x\in O$.
	\end{enumerate}
    The last statement \ref{s.preseveringVMRT} implies that general members in $\cK$ are sent to general members in $\cK'$ by $\varphi$. Denote by $\Phi:X\dashrightarrow \Bl_S(\bP^n)$ the rational map defined by $\varphi$. Let $D_0\subseteq D$ be the closed subvariety such that $\Phi$ is an isomorphism over $X\setminus D_0$. We note that $\Phi(D)$ is a divisor in $\Bl_S(\bP^n)$. In fact, a general minimal rational curve in $\cK$ is disjoint from the indeterminacy locus of $\Phi$ and it meets $D$ as $X$ has Picard number $1$. Thus, if $\Phi(D)$ has codimension $2$ in $\Bl_S(\bP^n)$, then every minimal rational curves in $\cK'$ passes through the codimension $2$ subvariety $\Phi(D)$, which is impossible. Thus $\Phi(D)$ is a divisor and this yields that the map $\Phi$ is a local isomorphism at general points of $D$. As a consequence, the closed subvariety $D_0$ is a proper subvariety of $D$ and hence has codimension at least $2$ in $X$ as $D$ is irreducible. On the other hand, the statement \ref{s.Eequivariantness} shows that the rational map $\Phi$ is $\bG_a^n$-equivariant and it follows that $D_0$ is $\bG_a^n$-stable.
    
    Next we consider the inverse map $\Phi^{-1}:\Bl_S(\bP^n)\dashrightarrow X$. Note that the points in the prime divisor $\widetilde{H}$ are fixed by $\bG_a^n$. We claim that $\Phi(\widetilde{H})$ has codimension at least $2$ in $X$. Otherwise, we must have $\Phi(\widetilde{H})=D$. In particular, the points in $D$ are fixed by $\bG_a^n$ and $X$ is isomorphic to the projective space $\bP^n$ by Proposition \ref{p.factsofEC}, which contradicts our assumption. Hence, the divisor $\widetilde{H}$ is contracted by $\Phi^{-1}$ and we have $\Phi^{-1}(E)=D$. In particular, the map $\Phi^{-1}$ is a local isomorphism at general points of $E$ and consequently there exists a closed proper $\bG_a^n$-stable subvariety $E_0$ of $E$ such that $\Phi^{-1}$ is an isomorphism over the Zariski open subset $\Bl_S(\bP^n)\setminus (\widetilde{H}\cup E_0)$.
\end{proof}

\subsection{Pseudoeffective cone of $\bP(T^*X)$}

In this subsection, we will finish the proof of the first statement of Theorem \ref{t.EC-main-thm}. Let $G$ be a connected linear algebraic group and let $X$ be an equivariant compactification of $G$. Fix a point $o$ in the unique open orbit. We define
\[
\fD\coloneqq\{\Delta\subseteq\bP(T^*X)\,|\, \Delta\ \text{is a}\ G\text{-stable}\ \pi\text{-horizontal prime divisor}\}
\]
and
\[
\fH\coloneqq\{\cH\subseteq \bP(T^*_o X)\,|\,\cH\ \text{is a reduced but maybe reducible hyperpsurface}\}.
\]
We can naturally identify $\fg^*$ to $T^*_o X$ via the cotangent map of the orbit map $\mu_o:G\rightarrow Go=O$ at the identity $e\in G$. In particular, we shall also regard the set $\fH$ as the set of reduced but maybe reducible hypersurfaces in $\bP(\fg^*)$. 

\begin{lem}
	\label{l.reconstruction-G-stable}
	Let $\Delta\in \fD$ and let $\cH\in \fH$ be the intersection $\Delta\cap \bP(T^*_o X)$. If $G$ is commutative, then $\Delta=D_{\cH}$.
\end{lem}

\begin{proof}
	Since both $\Delta$ and $D_{\cH}$ are $\pi$-horizontal, it is enough to show that the equality $\Delta=D_{\cH}$ holds over the open subset $\bP(T^*O)$. For an arbitrary point $o'\in O$, there exists a unique element $g\in G$ such that $o'=go$. Moreover, as $\Delta$ is $G$-stable, we must have
	\[
	\Delta\cap \bP(T^*_{o'} X) = d\mu_g|_o(d\mu_o|_e(\cH)),
	\]
	where $d\mu_g$ is the tangent map of the map $\mu_g:X\rightarrow X$, $x\mapsto gx$, at $x$. Consider the following diagram
	\[
	\begin{tikzcd}[column sep=large, row sep=large]
		G \arrow[r,"{\mu_o}"] \arrow[d,"\text{id}" left]
		& Go=O \arrow[d,"{\mu_g}"] \\
		G \arrow[r,"{\mu_{o'}}" below] 
		& Go'=O
	\end{tikzcd}
	\]
	Since $G$ is commutative, for any $g'\in G$, we have
	\[
	\mu_g(\mu_o(g'))=g(g'o)=(g'g)o=g'(go)=\mu_{o'}(g').
	\]
	In particular, the diagram above is commutative. This yields
	\[
	\Delta\cap \bP(T^*_{o'} X) = d\mu_g|_o\circ d\mu_o|_e (\cH) = d\mu_{o'}|_e (\cH) = D_{\cH}\cap \bP(T^*_{o'} X).
	\]
	Thus, we have $\Delta=D_{\cH}$ over $\bP(T^*O)$ and hence $\Delta=D_{\cH}$.
\end{proof}

The first statement in Theorem \ref{t.EC-main-thm} is a special case of the following more general result.

\begin{prop}
	\label{p.Pseff-cone-EC}
	Let $G$ be a connected commutative linear algebraic group, and let $X$ be a smooth equivariant compactification of $G$ with the unique open orbit $O$. Then the pseudoeffective cone $\Eff(\bP(T^*X))$ of $\bP(T^*X)$ is generated by following divisors:
	\begin{enumerate}
		\item the divisors $\pi^*D$, where $\pi:\bP(T^*X)\rightarrow X$ is the natural projection and $D$ is an irreducible component of the complement $X\setminus O$.
		
		\item the prime divisors $D_{\cH}$, where $\cH$ is an irreducible reduced hypersurface in $\bP(\fg^*)$.
	\end{enumerate}
\end{prop}

\begin{proof}
	The action of $G$ on $X$ can be naturally lifted to an action on $\bP(T^*X)$. Thus, according to Theorem \ref{t.pseffconeGvariety}, the pseudoeffective cone of $\bP(T^*X)$ is generated by $G$-stable prime divisors. Let $\Delta$ be a prime $G$-stable prime divisor in $\bP(T^*X)$. Then we have the following possibilities for $\Delta$:
	\begin{itemize}
		\item the divisor $\Delta$ is $\pi$-vertical;
		
		\item the divisor $\Delta$ is $\pi$-horizontal.
	\end{itemize}
    If the prime divisor $\Delta$ is $\pi$-vertical, then there exists a $G$-stable prime divisor $D$ in $X$ such that $\pi^*D=\Delta$ because the projection $\pi$ is $G$-equivariant. This implies that $D$ is an irreducible component of $X\setminus O$. 
   
    Now we assume that $\Delta$ is $\pi$-horizontal. Fix a point $o$ in the open orbit $O$. Taking intersection with $\bP(T^*_o X)$ yields an injection $\fD\rightarrow \fH$. Let $\Delta\in \fD$ be an arbitrary element and let $\cH\in \fH$ be the corresponding hypersurface. Then we have $\Delta=D_{\cH}$ by Lemma \ref{l.reconstruction-G-stable}.  Note that if $\cH=\cH_1+\cH_2$ as divisors in $\bP(\fg^*)$, then clearly we have $D_{\cH}=D_{\cH_1}+D_{\cH_2}$ as divisors in $\bP(T^*X)$. Thus we may assume that the hypersurfaces in $\fH$ are irreducible.
\end{proof}

\begin{rem}
	The commutativity of $G$ is necessary in Proposition \ref{p.Pseff-cone-EC}. In general, if $G$ is not commutative, then Lemma \ref{l.reconstruction-G-stable} is false and it is possible that there are $G$-stable prime divisors in $\bP(T^*X)$ which are not of the form $D_{\cH}$. For example, for the quintic del Pezzo threefold $X=V_3$ (cf. Example \ref{e.V5}), its total dual VMRT $\check{\cC}\subseteq \bP(T^*X)$ is an $\SL_2$-invariant prime divisor such that the intersection $\check{\cC}\cap \bP(T^*_x X)$ for a point $x$ contained in the open orbit is a union of three hyperplanes (see \cite[\S\,5]{HoeringLiuShao2020}). In particular, if $\check{\cC}$ is of the form $D_{\cH}$, then $\cH$ is also a union of three hyperplanes in $\bP(\fg^*)$ and consequently $D_{\cH}$ is a reducible divisor containing three irreducible components, which contradicts the irreducibility of $\check{\cC}$. Actually, the pseudoeffective cone of $\bP(T^*V_3)$ is generated by $\check{\cC}$ and $\pi^*D$, where $D$ is the closure of the unique $2$-dimensional $\SL_2$-orbit.
\end{rem}

\subsection{Pseudoeffective slope}

In this subsection, we will pursue further the study of the pseudoeffective cone of $\bP(T^*X)$ for $X$ being a smooth equivariant compactification with Picard number $1$ of a vector group. 

\begin{defn}[\protect{Pseudoeffective slope of vector bundles}]
	Let $E$ be a vector bundle over a normal projective variety $X$, and let $A$ be a big $\bR$-Cartier divisor on $X$. The pseudoeffective slope of $E$ with respect to $A$ is defined as
	\[
	\mu(E,A)\coloneqq \sup\{\varepsilon\in\bR\,|\,\Lambda-\varepsilon \pi^*A\ \text{is pseudoeffective}\},
	\]
	where $\pi:\bP(E^*)\rightarrow X$ is the natural projection and $\Lambda$ is the tautological divisor of $\bP(E^*)$.
\end{defn}

\begin{rem}
	The invariant $\mu(E,A)$ is also called \emph{pseudoeffective threshold of $E$ with respect to $A$} in \cite{Shao2020,FuLiu2021} and $E$ is big if and only if $\mu(E,A)>0$ for some $A$ on $X$. On the other hand, if $X$ is projective manifold with Picard number $1$, then the pseudoeffective cone of $\bP(T^*X)$ is generated by $\Lambda -\mu(E,A)\pi^*A$ and $\pi^*A$, where $A$ is an ample divisor on $X$.
\end{rem}

\subsubsection{Behavior under deformation}

Let $p:\cX\rightarrow \Delta$ be smooth family of projective manifolds over a disk $\Delta$. By Semicontinuity Theorem, if the tangent bundle of the fibre $\cX_t$ is big for $t\not=0$, then the tangent bundle of the central fibre $\cX_0$ is also big. Thus one may expect to get more examples of Fano manifolds of Picard number $1$ and with big tangent bundle by degenerating known examples. Nevertheless, it turns out that this may be not so successful. Recall that a smooth projective $X$ is \emph{globally rigid} if for any smooth deformation $\cX\rightarrow \Delta$ with $\cX_t\cong X$ for any $t\not=0$, we have $\cX_0\cong X$.
\begin{itemize}
	\item The rational homogeneous spaces of Picard number $1$ are globally rigid except the orthogonal Grassmannian $B_3/P_2=\Gr_q(2,7)$ by a series works of Hwang and Mok \cite[Main Theorem]{HwangMok2005} and the latter one has a degeneration to $X^5$ \cite[Proposition 2.3]{PasquierPerrin2010}.
	
	\item The odd symplectic Grassmannians $X^3(m,i)$ $(m\geq 2, 1\leq i\leq m-1)$ are globally rigid by \cite[Theorem 1.7]{HwangLi2021}.
	
	\item the codimension $k(\leq 3)$ linear section $V_k$ of $\Gr(2,5)\subseteq \bP^9$ and the smooth $\bP^4$-general linear section $S_k^a$ of $\bS_5\subseteq \bP^{15}$ with $1\leq k\leq 3$ are globally rigid by the classification of Fano manifolds with coindex at most $3$.
\end{itemize}

\begin{question}[\protect{\cite{KimPark2019,Manivel2020a}}]
	Let $X$ be either a non-homogeneous smooth projective symmetric variety of Picard number $1$ or a non-homogeneous smooth projective horospherical variety of Picard number $1$, different from the odd symplectic Grassmannians. Is $X$ globally rigid?
\end{question}

 Conversely, if the tangent bundle of $\cX_0$ is big, then it is not clear for us if $T_{\cX_t}$ is big for $t$ small enough and up to our knowledge there are no known counter-examples yet. However, in certain special cases, we can show that the bigness of tangent bundles is preserved under small deformation.

\begin{lem}
	\label{l.defect-deformation}
	Let $E\rightarrow \Delta$ be a vector bundle over the disk $\Delta$ and let $\cS\subseteq \bP(E^*)$ be a smooth family of embedded smooth projective varieties over $\Delta$. Then the codegrees and the dual defects of the fibres $\cS_t\subseteq \bP(E^*_t)$ are independent of $t$.
\end{lem}

\begin{proof}
	Let us consider the total conormal variety $\cI\subseteq \bP(E^*)\times_{\Delta} \bP(E)$ of $\cS$; that is, the variety defined as follows:
	\[
	\{(t,s,[H])\,|\,t\in \Delta, s\in \cS_t, [H]\in \bP(E_t)\ \text{is a hyperplane tangent to}\ \cS_t\ \text{at}\ s\}
	\]
	Then the total dual variety $\check{\cS}\subseteq \bP(E)$ is the image of the natural projection $\cI\rightarrow \bP(E)$. Moreover, it is clear that the fibre of $\check{\cS}\rightarrow \Delta$ over $t$ is just the dual variety of the fibre of $\cS\rightarrow \Delta$ over $t$. As $\check{\cS}\rightarrow \Delta$ is flat, the degrees and the dimensions of the fibres of $\check{\cS}\rightarrow \Delta$ are independent of $t$ and so are the codegrees and the dual defects of the fibres of $\cS\rightarrow \Delta$.
\end{proof}

\begin{rem}
	The result is false without the smoothness assumption. This can be shown by considering a family of smooth hypersurfaces in $\bP^n$ degenerating to a dual defective singular hypersurface in $\bP^n$.
\end{rem}

\begin{prop}
	\label{p.bigness-deformation}
	Let $p:\cX\rightarrow \Delta$ be a smooth family of Fano manifolds with Picard number $1$. Let $\cK$ be an irreducible component of the relative Chow variety $\text{Chow}(\cX/\Delta)$ such that $\cK_t$ is a minimal rational component of $\cX_t$ for any $t\in \Delta$. Assume moreover that the VMRT of $\cK_t$ at general points of $\cX_t$ is smooth for every $t\in \Delta$. If the VMRT of $\cX_0$ is not dual defective and $T{\cX_0}$ is big, then $T{\cX_t}$ is big for any $t\in \Delta$ and we have
	\[
	\mu(T{\cX_t},-K_{\cX_t}) = \mu(T{\cX_0},-K_{\cX_0}),\ \forall t\in \Delta.
	\]
\end{prop}

\begin{proof}
	Let $\sigma:\Delta\rightarrow \cX$ be a general section passing through a general point in $\cX_0$. Then the normalized Chow space $\cK_{\sigma(t)}$ along this section gives a family of smooth projective varieties. On the other hand, since the VMRT of $\cX_t$ is smooth for any $t$, it follows that the VMRTs $\cC_{\sigma(t)}\subseteq \bP(T_{\sigma(t)}\cX_t)$ along $\sigma(\Delta)$ is a smooth family of embedded projective varieties. Then by Lemma \ref{l.defect-deformation}, the VMRT $\cC_{\sigma(t)}\subseteq \bP(T_{\sigma(t)} \cX_t)$ is not dual defective for any $t\in \Delta$. In particular, the  relative total dual VMRT $\check{\cC}_{\cX}\subseteq \bP(T^*(\cX/\Delta))$ of the relative total VMRT $\cC_{\cX}\subseteq \bP(T(\cX/\Delta))$ is a prime divisor, where $T(\cX/\Delta)$ is the relative tangent bundle of $p$. Since the fibration $\cX\rightarrow \Delta$ has relative Picard number $1$, there are two unique real numbers $a$ and $b$ such that 
	\[
	\check{\cC}_{\cX}\sim_p a\Lambda_{\cX} + b\pi^*K_{\cX/\Delta},
	\]
	where $\Lambda_{\cX}$ is the tautological divisor of $\bP(T^*(\cX/\Delta))$. Then it is clear that $a$ is equal to the codegree of the VMRT of $\cX_0$ (cf. Lemma \ref{l.defect-deformation}) and 
	\[
	\mu(T{\cX_t},-K_{\cX_t})=\frac{b}{a},\ \forall t\in \Delta.
	\]
	As $T{\cX_0}$ is big, we have $b>0$ by Theorem \ref{t.bigness-total-dual-VMRT}. Hence, the tangent bundle $T{\cX_t}$ is big for any $t\in \Delta$.
\end{proof}

\begin{rem}
	Recall that a smooth projective variety $X$ is said \emph{locally rigid} if for any smooth deformation $\cX\rightarrow \Delta$ with $\cX_0\cong X$, we have $\cX_t\cong X$ for $t$ in a small analytic neighbourhood of $0$. 
	\begin{itemize}
		\item Smooth projective horospherical varieties with Picard number $1$ and the smooth projective two-orbits varieties $\textbf{X}_1$ and $\textbf{X}_2$ are locally rigid except the horospherical $G_2$-variety $X^5$ \cite[Theorem 0.5 and Proposition 2.3]{PasquierPerrin2010}.
		
		\item Smooth projective symmetric varieties with Picard number $1$ are locally rigid, see \cite{KimPark2019,BaiFuManivel2020,Manivel2020a}.
		
		\item Smooth equivariant compactifications of vector groups  with Picard number $1$ may be not locally rigid. Among all the known examples (see Example \ref{e.Examples-EC-Picard-number-one}), the only locally non-rigid ones are the smooth $\bP^4$-general linear sections $\bS_k^a$ of $\bS_5\subseteq \bP^{15}$ with codimension $k=2$ or $3$, see \cite{BaiFuManivel2020}.
	\end{itemize}
\end{rem}

\begin{cor}
	\label{c.linear-section-S10}
	Let $S_k$ be a smooth codimension $k$ linear section of $\bS_5\subseteq \bP^{15}$. Then $TS_k$ is big if $k\leq 3$.
\end{cor}

\begin{proof}
	Recall that the VMRT of $\bS_5$ is the Grassmannian $\Gr(2,5)\subseteq \bP^9$ in its Pl\"ucker embedding which is self-dual. In particular, its dual defect is $2$. Moreover, the VMRT of $S_k$ is a smooth codimension $k$ linear section of $\Gr(2,5)\subseteq \bP^9$. In particular, the VMRT of $S_k$ is dual defect if and only if $k=1$. Moreover, if $k\leq 3$, a codimension $k$ smooth $\bP^4$-general linear section $S_k^a$ of $\bS_5$ is an equivariant compactification of a vector group. As a consequence, if $k=2$ or $3$, by Theorem \ref{t.Examples-main-thm} and Proposition \ref{p.bigness-deformation}, the tangent bundle $TS_k$ is big. On the other hand, if $k=1$, then there is only one class of $S_1$ up to projective equivalence. Hence, the tangent bundle $TS_1$ is also big.
\end{proof}

\begin{rem}
	The variety $S_9$ is a smooth curve of genus $7$ and $S_8$ is a smooth K3 surface. In particular, their tangent bundles are even not pseudoeffective. For $S_6$ and $S_7$, their VMRTs are $0$-dimensional and it follows from \cite[Theorem 1.1]{HoeringLiu2021} that their tangent bundles are not big. Thus the only remaining unknown cases are $S_4$ and $S_5$. On the other hand, there are exactly two isomorphic classes of $S_2$. The special one $S_2^a$ is an equivariant compactification of $\bG_a^8$, while the general one $S_2^g$ is the $G_2\times \text{PSL}_2$-variety $\textbf{X}_2$ given in \cite[Theorem 0.2 and Definition 2.12]{Pasquier2009} (cf. \cite[Proposition 4.8]{BaiFuManivel2020}).
\end{rem}

In the following we apply Proposition \ref{p.Intro-criterion} to treat the two-orbits $F_4$-variety $\textbf{X}_1$. Let us give a brief geometric description of $\textbf{X}_1$ \cite[Proof of Proposition 2.2 and Definition 2.11]{Pasquier2009}. Set $G=F_4$. Let $G/H=O\subseteq X$ be the unique open $G$-orbit and by $Z$ its complement, which is the unique closed $G$-orbit. Then $Z$ has codimension $3$. Let $P$ be a parabolic subgroup of $G$ containing $H$ and minimal for this property. Then $R(P)\subseteq H$ and $P$ is the maximal parabolic subgroup $P(\omega_1)$ of $F_4$. Let 
\[
\varphi:G/H=O\rightarrow Y=G/P\]
be the natural projection. Let $F$ be an arbitrary fibre of $\varphi$. Then $P$ acts transitively over $F$. Denote by $Q$ the quotient $P/R(P)$. Then $Q$ is a semisimple group of type $C_3$ and $Q$ acts transitively over $F$. Moreover, the $Q$-variety $F$ has an equivariant compactification $\Gr(2,\mathbb{C}^6)\subseteq \bP(\wedge^2 \mathbb{C}^6)$ whose boundary divisor is a closed $Q$-orbit isomorphic to $\Gr_{\omega}(2,6)$.

\begin{prop}
	\label{p.X1}
	The tangent bundle of $\textbf{X}_1$ is big.
\end{prop}

\begin{proof}
	Firstly we show that the image $\cM_F^P\subseteq \fp^*$ has dimension $2\dim(F)-1$. Indeed, note that the $Q$-variety $F$ has complexity $0$ and rank $1$. It follows from Theorem \ref{t.KnopdimensionofMomentmap} that the variety $\cM_F^Q\subseteq \fq^*$  has dimension $2\dim(F)-1$. Let $\bar{\fh}$ be the Lie algebra of the image $\bar{H}$ of $H$ in $Q$ and let $\iota:\fq^*\rightarrow \fp^*$ be the natural inclusion induced by $P\rightarrow Q$. As $R(P)\subseteq H$, we have $\iota(\bar{\fh}^{\perp})=\fh^{\perp}$. Consider the following commutative diagram
	\[
	\begin{tikzcd}[column sep=large, row sep=large]
		T^*F=P\ast_H \fh^{\perp} \arrow[r,"{\Phi_F^P}"] \arrow[d,"\sigma" left]
		    & \fp^* \\
		T^*F=Q\ast_{\bar{H}} \bar{\fh}^{\perp} \arrow[r,"{\Phi_F^Q}" below] 
		    & \fq^* \arrow[u,"\iota" right]
	\end{tikzcd}
	\]
	where $\sigma$ is an isomorphism. This yields $\cM_F^P=\iota(\cM_F^Q)$ and consequently $\cM_F^P$ has dimension $2\dim(F)-1$.
	
	Next we show that the image $\cM_X^G(F)$ has dimension $\dim(X)+\dim(F)-1$. To see this, let us consider the following commutative diagram
	\[
	\begin{tikzcd}[column sep=large, row sep=large]
		N_{F/X}^* \arrow[r] \arrow[rd,"\nu"]
		    & T^*X|_F \arrow[r] \arrow[d,"{\Phi_X^G}"]
		        & T^*F \arrow[d,"{\Phi_F^P}"]  \\
		    & \fg^* \arrow[r,"\eta"] 
		        &  \fp^*
	\end{tikzcd}
	\]
	As $\eta(\cM_X^G(F))=\eta(\Phi_X^G(T^*X|_F))=\Phi_F^P(T^*F)=\cM_F^P$, it follows that we have
	\[
	\dim(\cM_X^G(F)) \leq \dim(\cM_F^P) + \dim(\fg^*) - \dim(\fp^*)=\dim(X)+\dim(F)-1.
	\]
	
	Finally, note that the $G$-variety $\textbf{X}_1$ has complexity $0$ and rank $1$. The image $\cM_X^G$ has dimension $2\dim(X)-1$ by Theorem \ref{t.KnopdimensionofMomentmap}. Hence, we obtain
	\[
	2\dim(X)-1=\dim(\cM_X^G) \leq \dim(Y)+\dim(\cM_X^G(F)) \leq 2\dim(X)-1.
	\]
	This implies that $\dim(\cM_X^G(F))=\dim(X)+\dim(F)-1$. Let $A$ be a prime ample divisor in $Y$. Then the closure of $\varphi^*A$ in $\textbf{X}_1$ is an ample prime divisor as $\textbf{X}_1$ has Picard number $1$. On the other hand, note that we have
	\[
	\dim(\cM_X^G(\varphi^*A)) \leq \dim(A) + \dim(\cM_X^G(F)) = 2\dim(X)-2.
	\]
	Hence, according to Proposition \ref{p.Intro-criterion}, the tangent bundle of $\textbf{X}_1$ is big.
\end{proof}

\begin{question}
	Are the tangent bundles of $S_4$ and $S_5$ big?
\end{question}

\subsubsection{Pseudoeffective slope of equivariant compactifications}

Let $\cH\subseteq \bP(\fg^*)$ be an irreducible reduced hypersurface defined by $\xi \in \Sym^m \fg$. Then we have:
\[
D_{\xi} = D_{\cH} + \sum \mult_{\pi^*D} (D_{\xi}) \pi^*D,
\]
where $D$ runs over all the prime divisor in $X$ such that $\bP(\cM_X^G(D))$ is contained in $\cH$ and $\pi:\bP(T^*X)\rightarrow X$ is the natural projection.

\begin{notation}
	Let $C$ be a smooth projective curve and let $E$ be a vector bundle of rank $n$ over $C$. Assume that there exists a non-zero map $\varphi:E\rightarrow V^r$, where $V^r$ is the trivial vector bundle of rank $r$ over $C$. Let $p:\bP(V^r)=C\times \bP^{r-1}\rightarrow \bP^{r-1}$ be the second projection and let $\xi$ be a homogeneous polynomial over $\bP^{r-1}$. For a fibre $F$ of $\bP(E)\rightarrow C$, the \emph{multiplicity $m_F(\varphi,\xi)$ along $F$} is defined as the multiplicity of the pull-back $(p\circ \varphi)^*\xi$ along $F$.
\end{notation}

\begin{lem}
	\label{l.multiplicity}
	Let $D$ be a prime divisor in $X$ such that $\bP(\cM_X^G(D))$ is contained in $\cH$. Let $C$ be an irreducible curve not contained in $D$ and meeting a general point of $D$. Denote by $f:\widetilde{C}\rightarrow C$ its normalisation. Then we have
	\[
	\mult_{\pi^*D} (D_{\xi}) = m_F(\varphi,\xi),
	\]
	where the map $\varphi:f^*T^*X\rightarrow \widetilde{C}\times \fg^*$ is naturally induced by the moment map $\Phi_X^G$ and $F$ is a fibre of $\bP(f^*T^*X)\rightarrow \widetilde{C}$ over a point $c$ such that $f(c)\in D$.
\end{lem}

\begin{proof}
	This follows directly from the definition of $D_{\xi}$ and the fact that the multiplicity $\mult_{\pi^*D}(D_\xi)$ only depends on the multiplicity of $D_{\xi}$ along general fibres of $\pi^*D\rightarrow D$.
\end{proof}

Now we are in the position to finish the proof of Theorem \ref{t.EC-main-thm}.

\begin{proof}[Proof of Theorem \ref{t.EC-main-thm}]
	The first statement follows from Proposition \ref{p.Pseff-cone-EC}. For the second statement, as the VMRT of $X$ is smooth, following the notations in Proposition \ref{p.boundary-divisor-EC}, there exist closed subvarieties $E_0\subseteq \Bl_S(\bP^n)$ and $D_0\subseteq X$ of codimension at least $2$ such that the following morphism 
	\[
	\Phi:X\setminus D_0 \rightarrow \Bl_S(\bP^n)\setminus(\widetilde{H}\cup E_0)
	\]
	is an isomorphism. Moreover, the induced rational map $E\dashrightarrow D$ is birational and $\bG_a^n$-equivariant. Note that the $\bG_a^n$-orbits on $E=\bP(N_{S/\bP^n})\setminus W=\bP(N_{S/H})$ are just the fibres of the natural projection $E\setminus W\rightarrow S$ (see Example \ref{e.Examples-EC-structure}). Thus the general $\bG_a^n$-orbits on $D$ have dimension $n-p-1$, where $p$ is the dimension of the VMRT $S\subseteq \bP^{n-1}$ of $X$.

	Fix a point $o\in X$. If $l$ is a general minimal rational curve passing through $o$, then $l$ is the strict transform of a line $l'$ in $\bP^n$ passing through $o$. Moreover, we may also assume that the strict transform of $l$ in $\Bl_S(\bP^n)$ is disjoint from $\widetilde{H}\cup E_0$. In particular, the curve $l$ meets $D$ at a smooth point $z\in D\setminus D_0$. Since $l$ is standard, we have
	\[
	f^*TX \cong \sO_{\bP^1}(2)\oplus \sO_{\bP^1}(1)^{\oplus p} \oplus \sO_{\bP^1}^{\oplus n-p-1}.
	\]
	where $f:\bP^1\rightarrow l\subseteq X$ is the natural embedding. Denote by $T^+_lX$ the positive factor of $f^*TX$ and by $T_o^+ l$ the fibre of $T^+_l X$ over $o$. By Proposition \ref{p.factsofEC}, there exists a $1$-dimensional subspace $V_l$ of $\bG_a^n$ such that $l$ is the closure of the $V_l$-orbit of $o$. Moreover, the subbundle $T^+_l X$ is preserved by the $V_l$-action.
	
	Denote by $\fg_l$ the subspace of $\fg$ corresponding to the subspace $T_o^+l\subseteq T_o X\cong \fg$. Then the $V_l$-action induces a map of vector bundles
	\[
	\Psi: l\times \fg \rightarrow f^*TX
	\]
	such that the induced map $l\times \fg_l\rightarrow T^+_l X$ is non-degenerate along $l\setminus\{z\}$. Moreover, as $\bG_a^n$ is commutative, the dual map $\Psi^*:f^*T^*X\rightarrow l\times \fg^*$ is exactly the restriction of the map 
	\[
	\pi\times \Phi_X^{\bG_a^n}: T^*X \rightarrow X\times \fg^*.
	\]
	From the splitting type of $T^+_l X$, one can derive that the linear map $\{z\}\times \fg_l \rightarrow T_z^+ l$ is zero. As the $\bG_a^n$-orbit of $z$ has dimension $n-p-1$, the rank of the linear map 
	\[
	\{z\}\times \fg\rightarrow T_z X
	\]
	is $n-p-1$. This implies $\Phi_X^G(T^*_z X)=\fg_l^{\perp}$. On the other hand, as $\bP(T_o^+ l)\subseteq \bP(T_o X)$ is the projectivised tangent bundle of the VMRT $\cC_o\subseteq \bP(T_o X)$ at $[T_o l]$, thus we may regard $\fg_l^{\perp}$ as the set of hyperplanes in $\bP(T_o X)$ which are tangent to $\cC_o$ at $[T_o l]$. 
	
	For a general point $z\in D$, the $\bG_a^n$-orbit $O_z$ of $z$ is the image of a fibre of $E\setminus W\rightarrow S$ over some point $s\in S$. In particular, the strict transform of the line connecting $o$ and $s$ is a minimal rational curve $l$ on $X$ passing through $o$ and meeting $O_z$ at a point $z'$. As $G$ is commutative, we have $\Phi_X^G(T^*X|_{O_z})=\Phi_X^G(T_{z'}^*X)=\fg_l^{\perp}$, where $\fg_l$ is the subspace of $\fg$ corresponding to $T^+_o l$. As a consequence, the image $\bP(\cM_D^G)\subseteq \bP(\fg^*)$ is the closure of the following 
	\[
	\bigcup_{[l]\in \cK_x\ \text{general}} \bP(\fg_l^{\perp}) = \bigcup_{[l]\in \cK_x\ \text{general}} \bP((T^+_o l)^{\perp}) \subseteq \bP(T_o^* X)
	\]
	which is exactly the dual variety of the VMRT $\cC_o\subseteq \bP(T_o X)$ by definition.
	
	Finally, we assume that the VMRT is smooth and not dual defective. Then $\bP(\cM_D^G)$ is a hypersurface in $\bP(\fg^*)$. For simplicity, we denote it by $\cH$. Now we want to determine the cohomological class of the divisor $D_{\cH}\subseteq \bP(T^*X)$. The equality $D_{\cH}=\check{\cC}$ can be easily derived from our argument above (see also Lemma \ref{l.inclusion-dual-VMRT}). Hence, by Theorem \ref{t.bigness-total-dual-VMRT}, it remains to prove $D_{\cH}\sim a\Lambda - 2\pi^*D$.
	
	Choose a general minimal rational curve $l$ on $X$ meeting $D$ at $z$. Fix a general point $o\in l$ and identify $\cH\subseteq \bP(\fg^*)$ to the dual variety of $\cC_o\subseteq \bP(T_o X)$. Consider the following map $\Psi^*:f^*T^*X \rightarrow l\times \fg^*$. Denote by $\xi \in \Sym^a \fg$ a defining equation of $\cH$, where $a$ is the degree of $\cH$, i.e., the codegree of $\cC_o$. By Lemma \ref{l.multiplicity}, we only need to calculate $\mult_F(\Psi^*,\xi)$, where $F$ is the fibre of $f^*T^*X\rightarrow l$ over $z$. Fix a coordinate $t$ around $z\in \bA^1\subseteq l$. Then after choosing suitable trivialisation of $f^*T^*X$, the map $\Psi^*:\bA^1\times \bC^{n}\rightarrow \bA^1\times \bC^{n}$ can be written in coordinates as follows:
	\[
		(x,v_0,v_1,\dots,v_p,v_{p+1},v_{n-1})
		    \mapsto (x,t^2 v_0,tv_1,\dots,tv_p,v_{p+1},\cdots,v_{n-1}),
	\]
	where the first coordinate $v_0$ corresponds to the cotangent bundle $\sO_{\bP^1}(-2)$ of $l$ and the first $p+1$ coordinates correspond to the negative factors of $f^*T^*X$. Given a general point $y$ on $\bar{\Phi}_X^G(\bP(F))\subseteq \cH$, then we may assume that $\cH$ is smooth at $y$ as $l$ is general. In particular, by biduality theorem, the projectivised tangent bundle of $\cH$ at $y$ corresponds to the point $[T_o l]\in \cC_o$. This implies that the linear part of the local equation of $\cH$ at $y$ only consists of the first coordinate $v_0$. In particular, the local description above shows that multiplicity $\mult_F(\Psi^*,\xi)$ is $2$. Hence, we have $D_{\xi}=D_{\cH}+2\pi^*D$ and the result follows as $D_{\xi}\in |a\Lambda|$.
\end{proof}

\begin{rem}
	Let us give a more geometric description of the linear map $\{z\}\times \fg\rightarrow T_z X$. Let $\pi_o:\bP^{n}\setminus \{o\}\rightarrow H$ be the projection from $o$. Then $\pi_o$ induces a natural projection 
	\[
	p_o:\fg\setminus\{0\} \xrightarrow{d\mu_o} \textbf{T}_o \bP^n\setminus \{o\}=\bP^n\setminus\{o\}\xrightarrow{\pi_o} H,
	\]
	where $\textbf{T}_{o}\bP^n$ is the projectivised tangent space of $\bP^n$ at $o$. For a general point $s\in S$, for a point $z'\in E_s\setminus W_s$, the Lie algebra $\fg_{z'}$ of the isotropy subgroup $G_{z'}$ of $z'$ is exactly the linear subspace of the inverse image $p_o^{-1}(\textbf{T}_s S)=\fg_l$, where $\textbf{T}_{s} S\subseteq H$ is the projectivised tangent bundle of $S$ at $s$. Thus the map $\{z\}\times \fg\rightarrow T_z X$ is the projection $\fg \rightarrow \fg/\fg_l$. 
\end{rem}

\subsubsection{Codegree of VMRT}

Now we proceed to calculate the codegree of the VMRT of equivariant compactifications $X$ of vector groups given in Example \ref{e.Examples-EC-Picard-number-one}. If $X$ is an irreducible Hermitian symmetric space, the pseudoeffective cone of $\bP(T^*X)$ and hence the value $\mu(TX,-K_X)$ are determined in \cite{Shao2020} and \cite{FuLiu2021}. In particular, if the VMRT is not dual defective, it turns out that its codegree is equal to the rank of $X$ in the sense of \cite[Definition 4.6]{Shao2020}. Here we remark that the definition of rank in \cite[Definition 4.6]{Shao2020} of $X$ is different from that given in Definition \ref{d.rank-complexity}. For the remaining non-homogeneous examples, we summarise the results in the following table:

\renewcommand*{\arraystretch}{1.6}
\begin{longtable}{|M{2.2cm}|M{4.6cm}|M{1.7cm}|M{0.9cm}|M{2.8cm}|}
	\caption{Known examples of non-homogeneous EC-structure}
	\label{table.codegree-EC}
	\\
	\hline
	
	$X$
	& VMRT 
	& embedding
	& defect
	& codegree
	\\
	\hline
	
	$X^3(m,m-1)$ $(m\geq 2)$
	& $\bP(\sO_{\bP^{m-1}}(-1)\oplus \sO_{\bP^{m-1}}(-2))$
	& $|\sO(1)|$
	& $0$
	& $m+1$
	\\
	\hline
	
	$V_2$
	& $\bP^1$
	& $|\sO(3)|$
	& $0$
	& $4$
	\\
	\hline
	
	$S^a_1$
	& $V_1$
	& $|\sO(1)|$
	& $1$
	& --
	\\
	\hline
	
	$S^a_2$
	& $V_2$
	& $|\sO(1)|$
	& $0$
	& $5$
	\\
	\hline
	
	$S^a_3$
	& $V_3$
	& $|\sO(1)|$
	& $0$
	& $10$
	\\
	\hline
\end{longtable}

\paragraph{\textit{Odd Lagrangian Grassmannians.}}

Let $\textbf{a}:=(a_0,\dots,a_r)$ be a sequence of integers such that $0\leq a_0\leq \dots\leq a_r$ with $a_r>1$. Denote by $E_m(\textbf{a})$ the following vector bundle over $\bP^m$:
\[
\bigoplus_{i=0}^r \sO_{\bP^m}(-a_i).
\]
Then the tautological linear bundle $\sO_{\bP(E_m(\textbf{a}))}(1)$ of $\bP(E_m(\textbf{a}))$ is globally generated and defines a morphism
\[
\Phi_m(\textbf{a}):\bP(E_m(\textbf{a})) \rightarrow \bP^{N(m,\textbf{a})}.
\]
This map is birational because $a_r>0$. Write $S_m(\textbf{a})$ for the image of this map. Note that if $a_0>0$, the morphism $\Phi_m(\textbf{a})$ is an embedding.

According to \cite[\S\,6]{HwangLi2021} (see also \cite[Proposition 3.5.2]{HwangMok2005}), the VMRT of the odd symplectic Grassmannian $X^3(m,i)$ is projectively equivalent to 
\[
S_{m-1}(1^{2m-2i-1},2)\subseteq \bP^{N(m,1^{2m-2i-1},2)}.
\]
In particular, the codegree of the VMRT of odd Lagrangian Grassmannians $X^3(m,m-1)$ can be derived from the following general result.

\begin{prop}
	The dual variety of the scroll $S_m(1^r,2)\subseteq \bP^{N(m,1^r,2)}$ is a hypersurface of degree $m+r+1$ if $m\geq r$.
\end{prop}

\begin{proof}
	Denote by $\Lambda$ the hyperplane section of $S_{m}(1^r,2)$. Firstly we recall that the projective variety $S_{m}(1^r,2)$ is isomorphic to the blowing-up of the projective space $\bP^{m+r}$ along a linear subspace $\bL\cong\bP^{r-1}$. Denote by $\mu:S_m(1^r,2)\rightarrow \bP^{m+r}$ the blowing-up and by $E$ the exceptional divisor. Then we have an isomorphism
	\[
	\sO_{S_m(1^r,2)}(\Lambda) \cong \mu^*\sO_{\bP^{m+r}}(2)\otimes \sO_{S_m(1^r,2)}(-E).
	\]
	In particular, taking push-forward yields a linear isomorphism
	\[
	p:|\Lambda|\rightarrow |\sO_{\bP^{m+r}}(2)\otimes \sI_{\bL}|=H_{\bL},
	\]
	where $\sI_{\bL}$ is the ideal sheaf of $\bL$. Denote by $\check{S}\subseteq |\Lambda|$ the dual variety of $S_{m}(1^r,2)$, i.e., the closure of singular elements. For a general point $[Q]\in \check{S}$, we may assume that the singular locus of $Q$ is not contained in $E$. This implies that the push-forward $p_*Q$ is a singular element in the linear system $|\sO_{\bP^{m+r}}(2)|$ which contains $\bL$.
	
	Conversely, as $r\leq m$, a general element in $H_{\bL}$ is smooth and a general singular element $[Q]\in H_{\bL}$ is a quadric hypersurface containing $\bL$ such that it is a cone with a single point $p\in \bP^{m+r}\setminus\bL$ as vertex and hence $p^{-1}[Q]$ is contained in $\check{S}$. As a consequence, the map $p$ induces a dominate map 
	\[
	\bar{p}:\check{S}\rightarrow \check{X}\cap H_{\bL}\subseteq |\sO_{\bP^{m+r}}(2)|,
	\]
	where $\check{X}$ is the dual variety of the Veronese embedding $X=\nu_2(\bP^{m+r})\subseteq |\sO_{\bP^{m+r}}(2)|$. As $\bP^{m+r}$ is homogeneous, the variety $\check{X}\cap H_{\bL}$ is an irreducible proper subvariety of $H_{\bL}$. In particular, the map $\bar{p}$ is an isomorphism. Note that $\check{X}$ is a hypersurface of degree $m+r+1$ by Boole formula \cite[Example 6.4]{Tevelev2005}, hence $\check{S}\subseteq |\Lambda|$ is a hypersurface of degree $m+r+1$.
\end{proof}

\begin{rem}
    Let $[Q]\in \check{S}$ be a general singular hyperplane section of $S$.  If $m<r$, then the divisor $p_*Q$ is a quadric cone containing $\bL$ with vertex $\bL'\subseteq \bP^{m+r}$, which is a $(r-m)$-dimensional linear subspace such that $\dim(\bL\cap \bL')=r-m-1$ In particular, the singular locus of $Q$ has dimension $r-m$ and this implies that the scroll $S_m(1^r,2)$ has dual defect $r-m$ \cite[Theorem 7.21]{Tevelev2005}.
\end{rem}

\paragraph{\textit{Linear section $V_k$ of the Grassmannian $\Gr(2,5)$.}}
The VMRT of the Grassmannian $\Gr(2,5)$ is projectively equivalent to the Segre embedding $\bP^1\times \bP^2\subseteq \bP^5$. Moreover, for $k\leq 3$, there is only one isomorphic class of codimension $k$ linear section $V_k$ of $\Gr(2,5)$. This implies that the VMRT of $V_k$ is projectively equivalent to a general linear section of $\bP^1\times \bP^2$ with codimension $k$. Then an easy computation shows that the VMRT of $V_2$ is the twisted cubic in $\bP^3$ whose dual variety is a quartic surface, see for instance \cite[Example 10.3]{Tevelev2005}.

\bigskip

\paragraph{\textit{Linear section $S_k$ of the spinor tenfold $\bS_5$.}}
The VMRT of the $10$-dimensional spinor variety $\bS_5$ is the Grassmannian $\Gr(2,5)\subseteq \bP^9$ in its Pl\"ucker embedding. Hence, the VMRT of the codimension $k$ linear section $S_k$ of $\bS_{5}$ is projectively equivalent to the smooth codimension $k$ linear section $V_k\subseteq \bP^{9-k}$ of the Grassmannian $\Gr(2,5)$. As $\Gr(2,5)\subseteq \bP^9$ has dual defect $2$, the linear section $V_k\subseteq \bP^{9-k}$ has dual defect $\max\{0,2-k\}$ \cite[Theorem 5.3]{Tevelev2005}. In the following we will compute the codegree of $Z=V_k$, $k=2$ or $3$, using the Katz-Kleiman formula \cite[Theorem 6.2]{Tevelev2005}:
\[
\text{codeg}(Z)=\sum_{i=0}^{\dim(Z)} (i+1) c_{n-i}(T^*Z)\cdot H^{\dim(Z)-i},
\]
where $H$ is the hyperplane section. To calculate the Chern classes of $Z$, firstly we write the total Chern classes of $\Gr(2,5)$ as
\[
\begin{pmatrix}
	1  &    &    &     \\
	5  & 12 &    &     \\
	11 & 30 & 25 &     \\
	15 & 35 & 30 & 33
\end{pmatrix}
\]
where the rows and columns are labelled from $0$, and the $(i,j)$-th element is the coefficient of the Schubert cycles $\sigma_{i,j}$. From the tangent sequence of $Z$ we have
\[
c(Z)\coloneqq c(V_k)=\frac{c(\Gr(2,5))}{(1+\sigma_1)^k},
\] 
where $\sigma_1\coloneqq \sigma_{1,0}$ denotes the ample generator of the Picard group, i.e., the hyperplane section. Using the Pieri's formula, that is, 
\[
\sigma_{a,b}\cdot \sigma_1 = \sigma_{a+1,b} + \sigma_{a,b+1},
\]
a routine computation then yields 
\[
c(V_2) = 
\begin{pmatrix}
	1  &     &       &     \\
	3  & 5   &       &     \\
	4  & 6   & 4     &     \\
	4  & 2   & \ast  & \ast
\end{pmatrix}
\]
and
\[
c(V_3) = 
\begin{pmatrix}
	1  &      &       &     \\
	2  & 3    &       &     \\
	2  & 1    & \ast  &     \\
	2  & \ast & \ast  & \ast
\end{pmatrix}
\]
Using again the Pieri's formula and our computations of the degrees of the Schubert classes we deduce that the codegree of $V_2$ and $V_3$ are $5$ and $10 $, respectively.

There is also an alternative more geometric way to see that the codegree of $V_2$ is $5$. Since the Grassmannian $\Gr(2,5)\subseteq \bP^{9}$ is self-dual, the dual variety of $V_k\subseteq \bP^{9-k}$, $k=1$ or $2$, is projectively equivalent to the image of $\Gr(2,5)$ under the projection $\pi_{\bL}:\bP^9\dashrightarrow \bP^{9-k}$ from a general linear subspace $\bL\subseteq \bP^9$ of dimension $k-1$, see for instance \cite[Theorem 5.3]{Tevelev2005}. As $\bL$ is general, the restriction of $\pi_{\bL}$ to $\Gr(2,5)$ is a birational morphism. Since $\Gr(2,5)$ is of degree $5$ and with dimension $6$, it follows that the dual variety of $V_2$ is a hypersurface in $\bP^7$ with degree $5$.

\begin{rem}
	\begin{enumerate}
		\item Let $X$ be a Fano manifold with Picard number $1$. Then the \emph{anti-canonical pseudoeffective slope $\mu(TX,-K_X)$ of $TX$} is bounded by the \emph{maximal slope} of $TX$ with respect to $-K_X$ (see \cite[Lemma 2.8]{FuLiu2021}). In particular, if $TX$ is semi-stable, then $\mu(TX,-K_X)$ is bounded by $1/\dim(X)$. Actually, it is expected that this should hold without the semi-stability assumption (cf. \cite[Conjecture 1.3]{FuLiu2021}). On the other hand, while the semi-stability of $TX$ is confirmed in many cases, \cite[Theorem 0.3]{Kanemitsu2021} says that the tangent bundles of the horospherical varieties $X^1(m)$ $(m\geq 4)$ and $X^4$ are not semi-stable. Thus it is natural and interesting to ask if their anti-canonical pseudoeffective slopes are (strictly) dominated by the reciprocal of their dimensions.
		
		\item For odd Lagrangian Grassmannians $X=X^3(m,m-1)$ $(m\geq 2)$, according to Table \ref{table.codegree-EC}, we have
		\[
		\mu(TX,-K_X)=\frac{2}{(m+1)(m+2)} < \frac{1}{\dim(X)}=\frac{2}{m(m+3)}.
		\]
		
		\item By \cite[Corollary 1.4]{Shao2020}, \cite[Theorem 1.14]{FuLiu2021} and Table \ref{table.codegree-EC} above, the anti-canonical pseudoeffective slope $\mu(TX,-K_X)$ of the varieties in Example \ref{e.Examples-EC-Picard-number-one} are determined except the hyperplane section $S_1$ of $\bS_5$. As the VMRT of $S_1$ is dual defective, maybe we need a different treatment.
	\end{enumerate}	
\end{rem}
	
\bibliographystyle{alpha}
\bibliography{bigEC}
\end{document}